\documentclass[preprint,3p,12pt]{elsarticle}

\usepackage{amsmath}
\DeclareMathOperator*{\argmin}{arg\,min}
\usepackage{amssymb}
\usepackage[english]{babel}
\usepackage{graphicx}
\usepackage{epsfig}
\usepackage{float}
\usepackage{subfig}
\usepackage{cleveref}
\usepackage{mathrsfs}
\usepackage{verbatim}

\usepackage{chngcntr}
\usepackage{apptools}
\usepackage{amsthm}
\newtheorem{thm}{Theorem}
\newtheorem{lem}{Lemma}
\newdefinition{rmk}{Remark}
\newproof{pf}{Proof}
\newproof{pot}{Proof of Theorem \ref{thm2}}

\numberwithin{thm}{section}
\numberwithin{lem}{section}

\def\B#1{\mbox{\boldmath{$#1$}}}

\newcommand{\onequart}{\mbox{$\frac{1}{4}$}}

\newcommand{\pd}{\partial}

\newcommand{\bx}{\B{x}}

\newcommand{\onehalf}{\mbox{$\frac{1}{2}$}}

\newcommand{\WW}{\mathcal{W}}

\def\be{\begin{equation}}
\def\ee{\end{equation}}
\def\ba{\begin{array}}
\def\ea{\end{array}}
\def\bea{\begin{eqnarray}}
\def\eea{\end{eqnarray}}
\def\beas{\begin{eqnarray*}}
\def\eeas{\end{eqnarray*}}
\newcommand{\bseq}{\begin{subequations}}
\newcommand{\eseq}{\end{subequations}}

\usepackage{color,soul}

\usepackage{empheq}

\newcommand{\norm}[1]{|\!| #1 |\!|}

\newcommand{\dx}{\,\text{d}x}

\renewcommand{\d}[1]{\,\text{d}#1}

\newcommand*\closedset[1]{%
   \,
   \vbox{%
     \hrule height 0.5pt
     \kern0.25ex
     \hbox{%
       \kern -0.3em
       \ifmmode#1\else\ensuremath{#1}\fi
       \kern -0.3em
     }
   }
   \,
}

\biboptions{numbers,sort&compress}

\journal{Computer Methods in Applied Mechanics and Engineering} 

\begin{document}

\begin{frontmatter}



\title{\large Unification of variational multiscale analysis and Nitsche's method, \\and a resulting boundary layer fine-scale model}

\author{Stein~K.F.~Stoter\corref{cor1}$^{a,b}$}
\ead{Stote031@umn.edu}
\author{Marco~F.P.~ten~Eikelder$^c$}
\ead{M.F.P.tenEikelder@tudelft.nl}
\author{Frits~de~Prenter$^d$}
\ead{F.d.Prenter@tue.nl}
\author{Ido~Akkerman$^c$}
\ead{I.Akkerman@tudelft.nl}
\author{E.~Harald~van~Brummelen$^d$}
\ead{E.H.v.Brummelen@tue.nl}
\author{Clemens~V.~Verhoosel$^d$}
\ead{C.V.Verhoosel@tue.nl}
\author{Dominik~Schillinger$^{a,b}$}
\ead{Schillinger@ibnm.uni-hannover.de}

\address{$^a$ Department of Civil, Environmental, and Geo-$\,$Engineering, University of Minnesota, Minneapolis, USA \\
$^b$ Institute of Mechanics and Computational Mechanics, Leibniz University Hannover, Hannover, Germany \\
$^c$ Department of Mechanical, Maritime and Materials Engineering, Delft University of Technology, Delft, The Netherlands \\
$^d$ Department of Mechanical Engineering, Eindhoven University of Technology, Eindhoven, The Netherlands 
}

\cortext[cor1]{Corresponding author;\\
Department of Civil, Environmental, and Geo-$\,$Engineering,  University of Minnesota, 500 Pillsbury Drive S.E., Minneapolis, MN 55455-0116, USA.}

\begin{abstract}
We show that in the variational multiscale framework, the weak enforcement of essential boundary conditions via Nitsche's method corresponds directly to a particular choice of projection operator. 
The consistency, symmetry and penalty terms of Nitsche's method all originate from the fine-scale closure dictated by the corresponding scale decomposition. 
As a result of this formalism, we are able to determine the exact fine-scale contributions in Nitsche-type formulations.
In the context of the advection-diffusion equation, we develop a residual-based model that incorporates the non-vanishing fine scales at the Dirichlet boundaries. 
This results in an additional boundary term with a new model parameter. 
We then propose a parameter estimation strategy for all parameters involved 
that is also consistent for higher-order basis functions. 
We illustrate with numerical experiments that our new augmented model mitigates the overly diffusive behavior that the classical residual-based fine-scale model exhibits in boundary layers at boundaries with weakly enforced essential conditions.
\end{abstract}


\end{frontmatter}

\newpage

\tableofcontents


\newpage

\section{Introduction}
\label{sec:intro}

The variational multiscale (VMS) method was established in the 1990s by Hughes and coworkers as a universal framework for developing and classifying stabilized methods \cite{Hughes1995, Hughes1996, Hughes1998,Franca2006,Coley2018}. It was hypothesized that the unresolved fine-scale nature of the solution of the partial differential equation is of key importance for the stability of finite-element schemes. The VMS methodology offers a means to ascertain the effect of the fine-scale solution onto the resolved finite element solution. As exact expressions for the fine scales are often not available, these fine-scale effects must be modeled.
A particularly prevalent class of fine-scale models is that of the residual-based models \cite{Bazilevs2007}. By assuming vanishing fine scales on element boundaries, the associated fine-scale problem can be solved on each element locally. The fine-scale solution is then approximated as the residual of the coarse-scale finite element solution, multiplied by the averaged Green's function. It has been shown that the resulting stabilized formulation is closely related (and often equivalent) to classical stabilized methods \cite{Brezzi1997b,Hughes2004b}. Typical examples include streamline-upwind Petrov-Galerkin (SUPG) \cite{Brooks1982,Brezzi1997b,Hughes2004b,tenEikelder2018}, Galerkin least-squares (GLS) \cite{HUGHES1989173,Hughes2004b, tenEikelder2018}, and pressure-stabilized Petrov-Galerkin (PSPG) methods \cite{Tezduyar1991}. More recently, residual-based modeling of the fine scales has found its use as an effective turbulence model for finite element implementations of the Navier-Stokes equations \cite{Bazilevs2007,Codina2007,Masud2006,Chang2012,Wang2010, Gravemeier2011,Takizawa2014, tenEikelder2018ii}. Since this turbulence model is mathematically inspired by the fine-scale equations, its parameters are clearly defined and it yields consistent formulations. Both these points are in contrast with typical eddy viscosity models, which are phenomenologically inspired and variationally inconsistent.

Another approach that has been shown to yield favorable results for fluid-mechanics applications is the use of weakly imposed Dirichlet boundary conditions \cite{Bazilevs2007weak_a,Bazilevs2007weak_b, bazilevs2010isogeometric}. Typically, Nitsche's method is the method of choice for weakly enforcing essential boundary conditions. While Nitsche's method was initially proposed in relation to energy minimization functionals \cite{Nitsche1971}, both its symmetric and nonsymmetric variants have since been studied extensively in fluid-mechanics applications \cite{Bazilevs2007weak_a,Bazilevs2007weak_b, bazilevs2010isogeometric,Burman2012}. One of the main drivers for the significant recent interest in weakly enforced boundary conditions is their importance in immersed finite element methods. Notable references in the context of fluid mechanics include \cite{bazilevs2012isogeometric, xu2016tetrahedral, hsu2016direct,kamensky2015immersogeometric, wu2017optimizing, hsu2014fluid,hoang2019}.
For immersed finite element methods, the approximation space is no longer tailored to fit the domain boundary. The essential boundary conditions can thus not easily be satisfied strongly. Hence, there is a need for weakly enforcing the Dirichlet boundary condition in the weak formulation. 


At first glance, the variational multiscale method and Nitsche's method appear to be at conflict: the basis for a variational multiscale decomposition is a well-posed continuous weak formulation, but Nitsche's method involves flexible spaces at Dirichlet boundaries and requires penalty terms that become unbounded in the continuous limit. Additionally, the fine-scale solution does, by design, not vanish on the Dirichlet boundary, which violates one of the key assumptions on which traditional residual-based fine-scale models are built. In previous work, we focused on discontinuous Galerkin methods, where the discontinuities between elements give rise to similar issues \cite{Stoter2017a,Stoter2017b,Stoter2019a}. The goal of this article is to completely eliminate these issues for weak boundary imposition in Nitsche-type formulations.

The remainder of this article is structured as follows. In \cref{sec:VMSform}, we derive a variational multiscale finite element formulation of the advection-diffusion equation and we show that Nitsche's method arises from a particular choice of fine-scale closure. In \cref{sec:VMSmodel}, we develop the fine-scale model that takes into account the non-vanishing fine scales at the Dirichlet boundary and provide estimates for the involved model parameters. The complete formulation is summarized in \cref{ssec:summary}. Next, in \cref{sec:coerc}, we show that the resulting bilinear form is coercive. In \cref{sec:numexpP1}, we verify the theory for a one-dimensional model problem, and in \cref{sec:numexpHO} we computationally investigate the performance for a two-dimensional model problem that involves multiple boundary layers. 
In \cref{sec:concl}, we present concluding remarks. \\[-0.6cm] 

\section{A variational multiscale derivation of the finite element formulation }
\label{sec:VMSform}
The classical model problem for variational multiscale analysis is the steady advection-diffusion equation. Let $\Omega$ denote the spatial domain with boundary $\partial \Omega$. The governing equations in strong form read:\\[-0.6cm]
\begin{subequations}
  \label{Strong}
  \begin{alignat}{3}
    \B{a} \cdot\nabla \phi -\nabla \cdot \kappa\nabla \phi &= f && \quad \text{in} \quad &&\Omega,\label{PDE}\\
    \phi &= \phi_D && \quad \text{on} \quad &&\partial\Omega_D,\label{Dirichlet}\\
    \kappa \, \partial_n \phi &= g_N && \quad \text{on} \quad &&\partial\Omega_N^+,\label{NeumOut}\\
    \kappa \, \partial_n \phi - \B{a} \cdot \B{n} \phi &= g_N && \quad \text{on} \quad &&\partial\Omega_N^-,\label{NeumIn}
  \end{alignat}
\end{subequations}
where the dependent variable $\phi=\phi(\bx)$ maps $\Omega$ into $\mathbb{R}$. The source function $f: \Omega \rightarrow \mathbb{R}$, the Dirichlet data $\phi_D:\partial\Omega_D \rightarrow \mathbb{R}$ and the Neumann (or Robin) data $g_N:\partial\Omega_N \rightarrow \mathbb{R}$ are exogenous functions that are assumed to be $L^2$-integrable on their respective domains. The advective velocity $\B{a}=\B{a}(\bx)$ is a given solenoidal vector field ($\nabla \cdot \B{a} = 0$) and the diffusivity $\kappa$ is strictly positive. 
 The Dirichlet and Neumann (or Robin) boundaries $\partial\Omega_D$ and $\partial\Omega_N$ are complementary subsets of the boundary
$\partial\Omega$, i.e. $\partial\Omega = \closedset{\partial\Omega_D} \cup \closedset{\partial\Omega_N}$, and the superscripts $+$ and $-$ indicate outflow ($\B{a} \cdot \B{n} \geq 0$) and inflow parts ($\B{a} \cdot \B{n}<0$) respectively. On the boundary, the normal gradient is denoted $\pd_n \phi = \B{n} \cdot\nabla \phi$. 
By convention, $\B{n}$ denotes the outward facing unit normal vector.

\subsection{Variational multiscale weak formulation}
\label{ssec:VMSform}

To obtain the weak formulation we multiply by a test function and integrate by parts wherever suitable. Different from classical functional (and variational multiscale) analysis we keep the traces of our function spaces on the domain boundary variable. This requires the use of Lagrange multipliers for the enforcement of the Dirichlet boundary conditions. To ensure inf-sup stability of the resulting bilinear form, we substitute the known data on the inflow part of the Dirichlet boundary in the advective term. 
We then obtain the following weak formulation:\\[-0.6cm]
\begin{subequations}
  \label{Weak}
  \begin{alignat}{1}
& \hspace{-0.5cm}\text{Find } \phi \in \mathcal{W} \text{ and } \lambda \in \mathcal{Q} 
\text{ s.t. } \forall\, w \in \WW \text{ and } q \in \mathcal{Q}:\nonumber\\
  \begin{split}
   & - \big( \B{a} \cdot\nabla w ,  \phi \big)_\Omega  
   + \big\langle  \B{a} \cdot \B{n}\,  w ,\phi\big\rangle_{\partial\Omega^+} 
   + \big( \nabla w, \kappa \nabla \phi \big)_\Omega 
   + \big\langle w , \lambda \big \rangle_{\partial\Omega_D}  \\
   &\hspace{15mm}
    = \big( w , f \big)_\Omega 
   + \big\langle w , g_N \big \rangle_{\partial\Omega_N}
      - \big\langle  \B{a} \cdot \B{n}\,  w ,\phi_D\big\rangle_{\partial\Omega_D^-}, \label{PDEWF}
 \end{split}\\
   &  \big\langle q , \phi \big \rangle_{\partial\Omega_D}  =  \big\langle q , \phi_D \big \rangle_{\partial\Omega_D} , \label{DirichletWF}
  \end{alignat}
\end{subequations}
where $\big(\,\cdot\,,\,\cdot\,\big)_\Omega$ denotes the $L^2$-inner product on domain $\Omega$, and $\big\langle\,\cdot\,,\,\cdot\,\big\rangle_{\partial \Omega}$ denotes the $L^2$-inner product on a surface, which is to be interpreted in the sense of a duality pairing. 
In its current form, the suitable functional spaces may be identified as $\WW=H^1(\Omega)$ and $\mathcal{Q} = H^{-1/2}(\partial \Omega_D)$. 
Equivalence of the strong and weak forms, \cref{Strong,Weak}, dictates that $\lambda = -\kappa \partial_n \phi$.  

The variational multiscale approach splits the trial solution and test function spaces into coarse and fine scales. The coarse scales live on the finite element grid, whereas the fine scales are determined via a model equation. This decomposition may be written as:
\begin{align}
  \WW = \WW^h \oplus \WW', \label{WhWp}
\end{align}
where $\WW^h$ is the space spanned by the finite-dimensional discretization and the fine-scale space $\WW'$ is an infinite-dimensional complement in $\WW$. The components of the solutions and test functions decouple as
\begin{subequations}\label{wwhwp}
  \begin{alignat}{1}
  \phi =& \phi^h + \phi',\\
  w =& w^h + w',
  \end{alignat}
\end{subequations}
with coarse scales $\phi^h, w^h \in \WW^h$ and fine scales $\phi', w' \in \WW'$. The direct sum decomposition in \cref{WhWp} is associated with a projection operator: 
\begin{subequations}
  \begin{alignat}{1}
  w^h =& \mathscr{P}^h w \in \WW^h, \label{eq:opt_proj}\\
  w' =& \left(\mathscr{I}-\mathscr{P}^h\right) w \in \WW', \label{wp}
  \end{alignat}\label{whwp}%
\end{subequations}
where $\mathscr{P}: \WW \rightarrow \WW^h$ is the projector and $\mathscr{I}: \WW \rightarrow \WW$ is the identity operator. Formally, this projector is incorporated in the weak formulation through the definition of the fine-scale space $\WW'$. It follows from \cref{wp} that $\WW'=\text{Im}(\mathscr{I}-\mathscr{P}^h) = \text{Ker}\mathscr{P}^h $. Then, the direct sum decomposition of \cref{WhWp} ensures the unique decomposition that satisfies \cref{wwhwp,whwp}. Using this multiscale split we arrive at the following alternative -- equivalent -- weak statement:
\begin{subequations}
  \begin{alignat}{1}
& \hspace{-0.25cm}\text{Find } \phi^h \in \mathcal{W}^h, \phi' \in \mathcal{W}', \lambda \in \mathcal{Q} \text{ s.t. } \forall w^h \in \WW^h, w' \in \WW', q \in \mathcal{Q}:\nonumber\\
  \begin{split}
   & - \big( \B{a} \cdot\nabla w^h , \phi^h \!+ \phi' \big)_\Omega  
   \!+ \big( \nabla w^h, \kappa \nabla \phi^h \!+ \kappa \nabla\phi' \big)_\Omega  
    \! + \big\langle  \B{a} \cdot \B{n}\,  w^h ,\phi^h \!+ \phi'\big\rangle_{\partial\Omega^+} 
   \! + \big\langle w^h , \lambda \big \rangle_{\partial\Omega_D} \\
     &\hspace{15mm} = \big( w^h , f \big)_\Omega 
   + \big\langle w^h , g_N \big \rangle_{\partial\Omega_N}
      - \big\langle  \B{a} \cdot \B{n}\,  w^h ,\phi_D\big\rangle_{\partial\Omega_D^-} ,
  \end{split}\label{eq:L_exact} \\
  \begin{split}
   & - \big( \B{a} \cdot\nabla w' ,  \phi^h + \phi' \big)_\Omega  
   + \big( \nabla w', \kappa \nabla\phi^h + \kappa \nabla\phi' \big)_\Omega 
   + \big\langle  \B{a} \cdot \B{n}\,  w' ,\phi^h + \phi'\big\rangle_{\partial\Omega^+} 
   \!+ \big\langle w' , \lambda \big \rangle_{\partial\Omega_D} \\
     &\hspace{15mm} = \big( w' , f \big)_\Omega 
   + \big\langle w' , g_N \big \rangle_{\partial\Omega_N}
      - \big\langle  \B{a} \cdot \B{n}\,  w' ,\phi_D\big\rangle_{\partial\Omega_D^-} ,
      \end{split}\label{eq:S_exact}\\
 &\big\langle q , \phi^h + \phi' \big \rangle_{\partial\Omega_D} =    
 \big\langle q , \phi_D \big \rangle_{\partial\Omega_D} \,. \label{lagrange2}
  \end{alignat}
\end{subequations}
Equation (\ref{eq:L_exact}) is the `coarse-scale problem', and can be interpreted as a relation for $\phi^h$ for a given $\lambda$ and $\phi'$. The Lagrange multiplier is taken care of by substituting the known value in the continuous case, viz. $\lambda = -\kappa \partial_n \phi^h -\kappa \partial_n \phi'$. 
Similarly, the `fine-scale problem' of \cref{eq:S_exact} can be conceived as a relation for the fine-scale component $\phi'\in\WW'$. This space, however, is infinite-dimensional and is thus not amenable to discrete implementation. Hence, in \cref{eq:L_exact} a closure model will be substituted in place of the fine-scale solution.

\subsection{Nitsche's method as a partial fine-scale closure}
\label{ssec:NitscheVMS}


Next, we show that a particular fine-scale closure condition leads to Nitsche's classical formulation. 
Our goal is to illustrate that Nitsche's method and the accompanying penalty terms are not in conflict with the VMS theory, but rather can be interpreted naturally in the VMS framework as a particular choice of fine-scale closure. 
Consider the following projection operator, which we will refer to as the Nitsche projector:
\begin{align}
\begin{alignedat}{3}
    \mathscr{P}_{\!N}:&\,\WW &&\rightarrow&& \,\,\WW^h\\
    & \,\,\phi &&\mapsto&& \,\,\argmin\limits_{\phi^h\in \WW^h} \int\limits_{\Omega} \frac{1}{2}\kappa (\nabla \phi - \nabla \phi^h )\cdot (\nabla \phi - \nabla \phi^h ) - \! \int\limits_{\partial \Omega_D} \kappa ( \pd_n \phi - \pd_n \phi^h ) (\phi- \phi^h) \\[-0.2cm]
    & && && \hspace{2cm}+\int\limits_{\partial \Omega_D} \frac{1}{2} \kappa \beta(\phi - \phi^h )^2  + \!\! \int\limits_{\partial \Omega^+} \frac{1}{2} \B{a}\cdot \B{n} (\phi - \phi^h )^2 .
\end{alignedat}\label{NitscheMinimization}
\end{align}
By taking the G\^ateaux derivative and subsequently replacing $\phi-\phi^h$ by $\phi'$, we obtain the following associated optimality condition for the fine scales:
\begin{align}
\begin{alignedat}{3}
    &- \big(\nabla v^h, \kappa \nabla \phi'\big)_{\Omega} + \big\langle v^h , \kappa \pd_n \phi' \big\rangle_{\partial\Omega_D}  + \big\langle  \kappa \pd_n v^h , \phi ' \big\rangle_{\partial\Omega_D} \\
    &\hspace{2cm}- \big\langle \kappa \beta v^h , \phi' \big\rangle_{\partial\Omega_D}  - \big\langle \B{a}\cdot\B{n}   v^h , \phi' \big\rangle_{\partial\Omega^+} = 0 \qquad \forall v^h \in \WW^h\,. \label{optimalitycond}
\end{alignedat}
\end{align}

Note that \cref{NitscheMinimization,optimalitycond} involve integrals of the normal derivative of functions in $\WW$ on the boundary of $\Omega$. In order for these integrals to be well-defined for all $w\in \WW$, we require a redefinition of $\WW$. If we set $\WW =\mathcal{V}\cup\mathcal{C}^0(\Omega)$ with $\mathcal{V} = \{\phi\in H^1(\Omega) : \Delta \phi \in L^2(\Omega)\}$, then:
\begin{align}
    \int_{\partial\Omega} \pd_n \phi \, v  = \int_{\Omega}\nabla\cdot ( \nabla \phi \, v ) = \int_{\Omega} \Delta\phi \, v + \int_{\Omega}\nabla\phi \cdot \nabla v \quad \forall\, v\in H^1(\Omega) 
\end{align}
defines $\pd_n \phi$ on the boundary for all $v\in\mathcal{V}$. Given that we are considering data of the form $f\in L^2(\Omega)$, the true solution $\phi$ will automatically be an element of this new space.

Typically, in the variational multiscale framework, one would attempt to invert the fine-scale problem of \cref{eq:S_exact} while satisfying the requirement posed by \cref{optimalitycond}. 
However, the particular structures of the coarse-scale problem and the Nitsche projector allow for a more direct inversion of (part of) the fine scales in the coarse-scale equation \cite{Stoter2017b,Stoter2019a}. First, we recognize that the fine-scale terms in \cref{optimalitycond} on the Dirichlet boundary may be written in terms of the coarse-scale solution through the definition $\phi'= \phi-\phi^h= \phi_D-\phi^h$. Then, since \cref{optimalitycond} holds for all $v^h \in \WW^h$, we may choose $v^h = w^h$ and add the obtained equality to the coarse-scale problem of \cref{eq:L_exact}. By following this procedure we are left with Nitsche's formulation of the advection-diffusion problem:
\begin{align}
& \hspace{-0.3cm}\text{Find } \phi^h \in \mathcal{W}^h \text{ s.t. } \forall\, w^h \in \WW^h:\nonumber\\
  \begin{split}
   & - \big( \B{a} \cdot\nabla w^h , \phi^h + \phi' \big)_\Omega  
       \!+ \big\langle  \B{a} \cdot \B{n}\,  w^h ,\phi^h \big\rangle_{\partial\Omega^+} 
   \!+ \big( \nabla w^h, \kappa \nabla \phi^h \big)_\Omega  - \big\langle w^h , \kappa \partial_n \phi^h  \big \rangle_{\partial\Omega_D} 
\\
   &\hspace{4mm}  
      \!- \big\langle  \kappa \pd_n w^h , \phi^h \big\rangle_{\partial\Omega_D}
   \!+ \big\langle w^h , \kappa \beta  \phi^h \big \rangle_{\partial\Omega_D}  
   = \big( w^h , f \big)_\Omega \! - \big\langle  \B{a} \cdot \B{n}\,  w^h ,\phi_D \big\rangle_{\partial\Omega_D^-}   
   \! + \big\langle w^h , g_N \big \rangle_{\partial\Omega_N} 
   \\
     & \hspace{8mm}
      \!- \big\langle  \kappa \pd_n w^h  , \phi_D \big\rangle_{\partial\Omega_D}
   \!+ \big\langle w^h ,  \kappa\beta \phi_D \big \rangle_{\partial\Omega_D}   ,
  \end{split}\label{MSNitsche}
\end{align}
where only one fine-scale term remains. We emphasize that both the symmetric term of Nitsche's method as well as its penalty term originate from the fine-scale inversion per \cref{optimalitycond}.

\textbf{Remark 1:} If we write the remaining fine-scale term as $\big(\mathcal{L}_{\text{stab}} w^h, \phi' \big)_\Omega$, then we observe that the appropriate fine-scale model (read: stabilization technique) in conjunction with Nitsche's method involves the SUPG operator $\mathcal{L}_{\text{stab}} = - \B{a} \cdot \nabla $. 
Typically, in variational multiscale stabilized methods, the adjoint differential operator $\mathcal{L}_{\text{stab}} = - \B{a} \cdot \nabla - \nabla\cdot \kappa \nabla $ is proposed. 
However, since the diffusive part of the fine-scale terms has already been incorporated via the introduction of the Nitsche terms, the diffusive operator no longer occurs in the fine-scale term.

\subsection{Analysis of the Nitsche projector}

The projector defines the scale decomposition of \cref{WhWp} via \cref{whwp}. This means that it impacts the appropriate modeling choices in the development of the fine-scale model. Before continuing the fine-scale modeling, we thus first dedicate a short study on the Nitsche projector. The ensuing analysis will heavily rely on the work by Hughes and Sangalli in \cite{Hughes2007}, who performed the same analysis using the $H^1_0$ and $L^2$ projectors. 

A projector may be defined by a finite number of functional constraints; as many as the dimension of $\WW^h$:
\begin{align}
\begin{alignedat}{3}
\mathscr{P}_{\!N} &\phi = \phi^h \text{ such that:}\\
&\mu_i\big(\phi-\phi^h\big)  = 0 \qquad i = 1,...,\dim(\WW^h)\,.
\end{alignedat}\label{PImage}
\end{align}
The $\mu_i(\cdot)$ functionals for the Nitsche projector may be inferred from the weak statement of \cref{optimalitycond}. By performing integration by parts on its first term while interpreting the resulting integral in the sense of distributions, the functional constraints follow from substituting the various candidates for $v^h$. For a one-dimensional domain with a set of nodes $\Gamma = \{x_1,x_2,..,x_n \}$ and element domains $\mathcal{T} = \{ [x_1,x_2],..,[x_{n-1},x_n] \}$ and a typical nodal finite element construction of $\WW^h$ with polynomial order $P$, we find:
\begin{subequations}
\begin{alignat}{7}
\mathscr{P}_{\!N} &\phi = \phi^h \text{ such that:}&&\nonumber\\
&\mu\big(\phi-\phi^h\big) = \phi(x_i)-\phi^h(x_i)  = 0 \quad && \hspace{-3cm}\text{for } x_i \in \Gamma\setminus  \partial\Omega_{D}\,, \label{FSSMa}\\
&\mu\big(\phi-\phi^h\big) = \textstyle\int_{K} (\phi-\phi^h) x^p  = 0  \quad && \hspace{-3cm} \text{for }K\in \mathcal{T} \text{ and } 0 \leq p \leq P\!-\!2\,,\label{FSSMd}\\
&\mu\big(\phi-\phi^h\big) = \kappa \pd_n (\phi-\phi^h)\big|_{x_i}  - \kappa\beta (\phi-\phi^h)\big|_{x_i}  = 0 \qquad&& \text{for } x_i\in \partial\Omega_{D}^-\,, \label{FSSMb}\\
&\mu\big(\phi-\phi^h\big) = \kappa \pd_n (\phi-\phi^h)\big|_{x_i} - (\kappa \beta + \B{a}\cdot\B{n} ) (\phi-\phi^h)\big|_{x_i} = 0  \qquad && \text{for } x_i\in\partial\Omega_{D}^+\,. \label{FSSMc}
\end{alignat}\label{PNImage}%
\end{subequations}
The first requirement dictates nodal exactness of the finite element formulation, and together with the second requirement these define the $H^1_0$ projector \cite{Hughes2007}. The last two originate from the extra degrees of freedom on the Dirichlet boundary. \newpage

\textbf{Remark 2:} The central role of the projector can be useful for the interpretation of the obtained finite element solution. For instance, for the current example of Nitsche's method, a better approximation of the true diffusive flux on the Dirichlet boundary could be obtained by rewriting \cref{FSSMb,FSSMc}:
\begin{subequations}
\begin{alignat}{2}
-\kappa \pd_n \phi &= -\kappa \pd_n \phi^h + \kappa\beta (\phi^h-\phi_D) && \quad \text{on}\quad \partial\Omega_{D}^-\,,\\
-\kappa \pd_n \phi &= -\kappa \pd_n \phi^h + (\kappa\beta+\B{a}\cdot\B{n}) (\phi^h-\phi_D)&& \quad \text{on}\quad \partial\Omega_{D}^+\,.
\end{alignat} \label{difflux}%
\end{subequations}
The same expressions were also proposed in \cite{Bazilevs2007weak_a} by Bazilevs \textit{et al.}, although motivated based on discrete conservation laws. 

\textbf{Remark 3:} In \cite{Harari2018}, Harari and Albocher perform a spectral analysis of Nitsche's formulation. Their work shows that its spectrum consists of \emph{i)} traditional modes that are independent of the Nitsche parameter and which vanish on Dirichlet boundaries, and \emph{ii)} modes that depend on the Nitsche parameter and are locally supported in a layer along the Dirichlet boundaries. 
In \cref{PNImage}, we observe a similar split in the (functional) constraints imposed by the Nitsche projector.

\section{A residual-based multiscale model with non-vanishing fine-scale boundary values}
\label{sec:VMSmodel}

For our finite element scheme to yield solutions close to $\phi^h = \mathscr{P}_{\!N}\, \phi$, the weighted integral of $\phi'$ in formulation (\ref{MSNitsche}) needs to be accurately modeled. The model for the remaining fine-scale quantity originates from the inversion of the fine-scale problem of \cref{eq:S_exact}. 

\subsection{Inversion of the fine-scale problem}
\label{ssec:InverseFSP}
Consider the following general form of the fine-scale problem:
\begin{alignat}{1}
& \text{Find } \phi' \in \WW' \text{ s.t. } \forall\, w' \in \WW':\nonumber \\
& \hspace{0.5cm}  a(w',\phi') = \big( w' , f \big)_\Omega - a(w',\phi^h)\,. 
\end{alignat}
By performing the appropriate integration by parts steps we obtain the following integral relation, to be interpreted in the sense of distributions:
\begin{align}
    \int_\Omega \mathcal{L}^* w' \, \phi' \d y \, +\, B(w',\phi';\partial\Omega) =  \int_\Omega w'\,(f-\mathcal{L}\phi^h) \d y =  \int_\Omega w'\,\mathcal{R}_{\phi^h} \d y \,,\label{GF1} 
\end{align}
where we denote our spatial variable $y$, for ease of notation later on. The operator $\mathcal{L}^*$ is the adjoint of the differential operator $\mathcal{L}=\B{a}\cdot\nabla-\nabla\cdot\kappa\nabla$, the coarse-scale residual is denoted $\mathcal{R}_{\phi^h}$, and $B(\cdot,\cdot\,;\partial\Omega)$ is a bilinear form that represents the boundary terms.

Then, we hypothesize that one can find a function $w'(y) \in \WW' $ such that $\mathcal{L}^*\! w'(y)$ acts on any function in $v\in\WW$ in the following way:
\begin{align}
\int_\Omega \mathcal{L}^*\! w' \,v  \d y =: \int_\Omega \mathcal{L}^*\! g'(x,y) \,v  \d y = \int_\Omega \delta_x v \d y + \sum\limits_{i=1}^{\dim(\WW^h)} c_{i}(x) \mu_i( v )\,, \label{gDef}
\end{align}
where we call this particular function the fine-scale Green's function $g'(x,y)$. The $\delta_{x}$ denotes the Dirac delta distribution at $x$, and the $\mu_i$ functionals refer to the functional constraints from \cref{PImage}. The sum provides a relaxation such that $g'(x,y)$ may be found in the constrained space $\WW'$.  From \cref{GF1} we obtain:
\begin{align}
\int_{\Omega} \mathcal{L}^*\! g'(x,y)\, \phi'(y) \,\text{d}y & = \int_{\Omega} \delta_x(y) \, \phi'(y) \,\text{d}y + \sum\limits_{i=1}^{\dim(\WW^h)} c_{i}(x)  \mu_i(\phi') \nonumber\\ 
& = \phi'(x) = \int_{\Omega} g'(x,y) \, \mathcal{R}_{\phi^h}  \,\text{d}y - B(g',\phi';\partial\Omega) \,,\label{phiinvert}
\end{align}
where the summation vanishes due to \cref{PImage}, and the term $B(g(x,y),\phi';\partial K)$ incorporates the fine-scale boundary conditions. This relation determines $\phi'$ from a given $\phi^h = \mathscr{P}_{\!N}\, \phi$.

\subsection{Adoption of the $H^1_0$ fine-scale Green's function}
\label{ssec:H10model}

The fine-scale Green's function in \cref{gDef,phiinvert} corresponds to the Nitsche projector. However, most literature on the variational multiscale method focuses on a scale decomposition by means of the $H^1_0$ projector. To maintain the connection with existing fine-scale models, we reintroduce the $H^1_0$ fine-scale Green's function as follows:
\begin{align}
g'_{N}(x,y) = g'_{H^1_0} (x,y) + \tilde{g}'(x,y)\,. \label{gdecompo}
\end{align}
The newly added subscripts indicate the projector with which the fine-scale Green's function is associated. The similarity between the $H^1_0$ projector and the Nitsche projector can be expressed in terms of their imposed functional constraints, from \cref{PImage}. The set of functions $\mu_i$ corresponding to the projector $\mathscr{P}_{\!H^1_0}$ is a subset of those of $\mathscr{P}_{\!N}$. If we order the set $\mu_i$ such that coinciding occurrences come first, then we may write for the $H^1_0$ fine-scale Green's function:
\begin{align}
\int_\Omega \mathcal{L}^*\! g'_{H^1_0}(x,y) \,v  \d y = \int_\Omega \delta_x v \d y + \hspace{-0.3cm}\sum\limits_{i=1}^{\dim(\WW^h\cap H^1_0)} \hspace{-0.3cm} d_{i}(x) \mu_i( v )\qquad \forall\, v\in \WW\,, \label{gDefH10}
\end{align}
which, after substitution in \cref{gDef}, gives a result for $\tilde{g}'(x,y)$:
\begin{align}
\int_\Omega \mathcal{L}^*\! \tilde{g}'(x,y)  \,v \d y = \hspace{-0.3cm}\sum\limits_{i=1}^{\dim(\WW^h\cap H^1_0)} \hspace{-0.3cm}(c_i(x)-d_i(x)) \mu_i(v) + \hspace{-0.5cm}\sum\limits_{i=\dim(\WW^h\cap H^1_0)+1}^{\dim(\WW^h)}\hspace{-0.5cm} c_{i}(x) \mu_i(v)\qquad \forall\, v\in \WW\,. \label{gDeftilde}
\end{align}
Substitution of \cref{gdecompo} into \cref{phiinvert} while using $ \mathcal{R}_{\phi^h} = \mathcal{L}\phi'$ gives:
\begin{align}
    \phi'(x) &= \int_{\Omega} \big( g'_{H^1_0} (x,y) + \tilde{g}'(x,y) \big)\, \mathcal{R}_{\phi^h}  \,\text{d}y - B( g'_{H^1_0} + \tilde{g}',\phi';\partial\Omega)\nonumber\\
    &= \int_{\Omega} g'_{H^1_0} (x,y) \mathcal{R}_{\phi^h}  \,\text{d}y - B( g'_{H^1_0} ,\phi';\partial\Omega)+ \int_{\Omega} \tilde{g}'(x,y) \mathcal{L}\phi' \,\d y - B( \tilde{g}',\phi';\partial\Omega)\nonumber\\
    &= \int_{\Omega} g'_{H^1_0} (x,y) \mathcal{R}_{\phi^h}  \,\text{d}y - B( g'_{H^1_0} ,\phi';\partial\Omega)+ \int_{\Omega}  \mathcal{L}^*\tilde{g}'(x,y)\phi' \,\d y \,, \label{derivphip}
\end{align}
where the last equality results from performing the integration by parts steps of \cref{GF1} in reverse. 

By using \cref{gDeftilde} in the last term of \cref{derivphip}, it follows from \cref{PImage} that this term vanishes when $\phi^h=\mathscr{P}_{N}\phi$. The expression for $\phi'$ becomes:
\begin{align}
    \phi'(x)     &= \int_{\Omega} g'_{H^1_0} (x,y) \mathcal{R}_{\phi^h}  \,\text{d}y - B( g'_{H^1_0} ,\phi';\partial\Omega)\,. \label{phiinvert2}
\end{align}

\textbf{Remark 4:} The inversion posed by \cref{phiinvert2} is no longer unique; it is satisfied for solutions $\phi^h = \mathscr{P}_{\!N}\, \phi$, but also for solutions $\phi^h = \mathscr{P}_{\!H^1_0}\, \phi$. However, the partial fine-scale closure discussed in \cref{ssec:NitscheVMS} is not satisfied by $\phi^h = \mathscr{P}_{\!H^1_0}\, \phi$. The formulation obtained after substitution of \cref{phiinvert2} into Nitsche's coarse-scale formulation of \cref{MSNitsche} will hence be uniquely satisfied by $\phi^h = \mathscr{P}_{\!N}\, \phi$.\\

Recall that the term $B(\cdot ,\cdot\,;\cdot)$ enforces the fine-scale boundary conditions. The original residual-based model assumes that the fine-scale solution vanishes on element boundaries~\cite{Bazilevs2007}. Since we aim not to make this assumption, this will be the key term that we retain to obtain a more suitable fine-scale model. In \cref{phiinvert2}, this term simplifies since $g'_{H^1_0}\in \WW'\cap H^1_0$ is zero on the boundary. We then restrict our analysis to partial differential equations for which we may write:
\begin{align}
B(g'_{H^1_0},\phi';\partial \Omega) = -\int_{\partial \Omega} \mathbb{H} g'_{H^1_0}(x,y)  \,\big(  \phi(y)-\phi^h(y) \big) \,\text{d}y \label{GFboundary},
\end{align}
with $\mathbb{H}$ being some differential operator. This includes the advection-diffusion operator, for which $\mathbb{H}$ is:
\begin{align}
\mathbb{H} = -\kappa\,\B{n}\cdot\nabla_y \label{H}.
\end{align}

The only 
term in the coarse-scale formulation of \cref{MSNitsche} in which $\phi'$ appears is the advective term.
We thus finally obtain the fine-scale contribution to the coarse-scale equation~as:
\begin{align}
\begin{split}
-\int_{\Omega} \B{a}\cdot\nabla w^h \, \phi' \dx &=  -\int_{\Omega}\! \int_{\Omega} \B{a}\cdot\nabla w^h(x)\,g'_{H^1_0}(x,y)\, \mathcal{R}_{\phi^h}(y) \,\text{d}y  \d{x} \\[-0.3cm]
&\hspace{-1cm}- \!\int_{\Omega}\int_{\partial \Omega} \B{a}\cdot\nabla w^h(x)  \mathbb{H} g'_{H^1_0}(x,y) \, \big(  \phi(y)-\phi^h(y) \big) \d{y} \d{x}\,.
\end{split}\label{eq: weighted fine-scale effect}
\end{align}
The unknown data $\phi(y)$ on the Neumann boundary, together with the double integration and the limited availability of Green's functions, make the closure relation of \cref{eq: weighted fine-scale effect}, albeit exact, unsuitable for computational use. Simplifications via approximations need to be introduced. The strategy that we will employ repeatedly in this article is to reformulate such that exactness is maintained in the case of constant physical parameters on a one-dimensional domain, while ease of implementation is established in the general case.

\subsection{The classical one-dimensional case}
\label{ssec:OneDModel}

In the one-dimensional case, the nodal exactness induced by the $\mathscr{P}_{\!H^1_0}$ projector results in an element-local fine-scale Green's function \cite{Hughes2007}. The double integrals in \cref{eq: weighted fine-scale effect} can thus be split in contributions of individual elements. The newly added term only affects elements that lie adjacent to the Dirichlet boundary, where the precise value of the fine-scale solution is known as $\phi'= \phi_D-\phi^h$. On Neumann boundaries the new term vanishes due to \cref{FSSMa}. Let us consider the contribution of one element that shares a node with the Dirichlet boundary:
\begin{align}
\begin{split}
-\int_{K} \B{a}\cdot\nabla w^h \, \phi' \dx &=-\int_{K}\! \int_{K} \B{a}\cdot\nabla w^h(x)\,g'_{H^1_0}(x,y)\, \mathcal{R}_{\phi^h}(y) \,\text{d}y  \d{x} \\[-0.3cm]
&\hspace{-1cm}- \!\int_{K}\int_{\partial K \cap\partial \Omega_D }\!\! \B{a}\cdot\nabla w^h(x)  \mathbb{H} g'_{H^1_0}(x,y) \, \big(  \phi_D(y)-\phi^h(y) \big) \d{y} \d{x}\,.  \label{1DK}
\end{split}
\end{align}
To simplify, we use the polynomial representation of the test function $w^h\big|_K$ and the residual $\mathcal{R}_{\phi^h}\big|_K$. When $\WW^h$ is constructed with $P$-order nodal elements and the source function $f$ is at most polynomial order $P-1$, then:
\begin{align}
    -\B{a}\cdot\nabla w^h(x)  = -a \sum\limits_{i=1}^{P}  \hat{w}_{i} \,i\,x^{i-1}\,; \qquad \mathcal{R}_{\phi^h}(y) = \sum\limits_{j=1}^{P}  \hat{R}_{j} \,y^{j-1}  \qquad \text{for }x\in K,\,y\in K  \,. \label{PolRep}
\end{align}
We also have the following properties of the fine-scale Green's function $g'_{H^1_0}(x,y)$:
\begin{subequations}
\begin{alignat}{3}
&\int_K \int_K x^q\, g'_{H^1_0}(x,y) \, y^r \text{d}y \d{x} = 0  \qquad&& \text{when } q < P-1 \text{ or } r < P-1 \, , \label{moment1}\\
&\int_K \int_{\partial K} x^q \, \mathbb{H} g'_{H^1_0}(x,y)  \text{d}y \d{x} = 0  \qquad&& \text{when } q < P \label{moment2} .
\end{alignat}\label{moments}%
\end{subequations}
Property \eqref{moment1} is shown in \cite{Hughes2007}, and we prove \eqref{moment2} in \ref{thm:0gamma}. Substituting \cref{PolRep} into \cref{1DK} while using \cref{moments} yields:
\begin{align}
\begin{split}
-\int_{K} \B{a}\cdot\nabla w^h \, \phi' \dx &=-\int_{K}\! \int_{K} a\, \big(\hat{w}_{P} \,P\,x^{P-1} \big) \,g'_{H^1_0}(x,y)\,\big(\hat{R}_{P} \,y^{P-1}\big)  \,\text{d}y  \d{x} \\[-0.3cm]
&\hspace{-1cm}- \!\int_{K}\int_{\partial K \cap\partial \Omega_D }\!\! a\, \big(\hat{w}_{P} \,P\,x^{P-1} \big)  \mathbb{H} g'_{H^1_0}(x,y) \, \big(  \phi_D(y)-\phi^h(y) \big) \d{y} \d{x} \,.
\end{split}
\end{align}
We assume $a$ to be constant in $K$, and we extract $a$ and all other constants from the double integration. We then integrate the two right hand side terms over $K$ and $F:=\partial K \cap\partial \Omega_D$ respectively, and we divide them by $|K|$ and $|F|$:
\begin{align}
\begin{split}
-\int_{K} \B{a}\cdot\nabla w^h \, \phi' \dx &=-\int_{K}\! a \, \hat{w}_{P} P\,h^{P-1} \Big[ \frac{1}{|K|}\int_{K}\! \int_{K}  \,\frac{x^{P-1}}{h^{P-1}} \,g'_{H^1_0}(x,y)\, \frac{y^{P-1}}{h^{P-1}}  \,\text{d}y  \d{x}\Big] \,\hat{R}_{P} h^{P-1} \d{\hat{x}}  \\
&\hspace{-1cm}-\int_{F }\!   a \, \hat{w}_{P} P \,h^{P-1} \Big[ \frac{1}{|F|}\int_{K}\int_{F }\!\!\,\frac{x^{P-1}}{h^{P-1}}  \mathbb{H} g'_{H^1_0}(x,y)\d{y} \d{x} \Big]  \, 
\big(  \phi_D(\hat{x})-\phi^h(\hat{x}) \big)  \d{\hat{x}} \,, \qquad\label{HOexact}
\end{split}\raisetag{0.8cm}
\end{align}
where we can identify the following model parameters:
\begin{subequations}
\begin{alignat}{4}
&\tau &&= \frac{1}{|K|} \int_{K} \int_{K} \,\frac{x^{P-1}}{h^{P-1}} \,g'_{H^1_0}(x,y)\, \frac{y^{P-1}}{h^{P-1}}  \,\text{d}y  \d{x}\,, \label{taudef}\\
&\gamma &&= \frac{1}{|F|} \int_{K}\!\int_{F} \frac{x^{P-1}}{h^{P-1}}  \mathbb{H} g'_{H^1_0}(x,y)\d{y} \d{x} \,. \label{gamdef} 
\end{alignat}\label{taugamdef}%
\end{subequations}
The multiplication and division by $h^{P-1}$ in \cref{HOexact} ($h$ being the element size) ensures that the parameters in \cref{taugamdef} remain dimensionally consistent with varying polynomial order.

\subsection{Fine-scale closure generalization}
\label{ssec:Generalization}

Up until now, all derivations have been exact. To make use of the integral expressions in \cref{HOexact} on multi-dimensional domains, we approximate them by the following inner products:
\begin{subequations}
\begin{alignat}{6}
    & -\int_K a \,\hat{w}_{P}P\,h^{P-1} \tau \hat{R}_{P} h^{P-1}\d{\hat{x}}  &&\approx  -\big(\B{a}\cdot \nabla w^h, \tau_{\text{eff}}  \mathcal{R}_{\phi^h} \big)_K \label{tautermapprx} \,, \\
    & - \int\limits_{F}  a \,\hat{w}_{P}P \,h^{P-1} \gamma \, ( \phi_D-\phi^h ) \d{\hat{x}} \,\, &&\approx  -\big\langle\B{a}\cdot \nabla w^h ,\gamma_{\text{eff}} \, ( \phi_D-\phi^h ) \big\rangle_{F^+ }\, . \label{gammatermapprx}
\end{alignat}\label{modelapproximations}%
\end{subequations}
As \cref{gammatermapprx} indicates, we only make use of the newly proposed term at the outflow Dirichlet boundary $F^+:=\partial K \cap\partial \Omega_D^+$. This is where the boundary layers occur, and where the weak enforcement of the Dirichlet conditions results in impactful fine-scale boundary values.

All the approximations involved in the final finite element formulation may be traced back to these two equations. Essentially, they shift the modeling effort onto the effective stabilization parameters $\tau_{\text{eff}}$ and $\gamma_{\text{eff}}$. 
We propose to design $\tau_{\text{eff}}$ and $\gamma_{\text{eff}}$ such that these are approximations of $\tau$ and $\gamma$ that take into account the change of (bi)linear forms, while being suitable for multi-dimensional computations for arbitrary order polynomial basis functions.

\subsubsection{Estimation of operator impact}
\label{ssec:OpEst}

In the one-dimensional case, the bilinear forms of the left-hand and right-hand sides of \cref{modelapproximations} may be written as: 
\begin{subequations}
\begin{alignat}{5}
& B_\text{vol}( w^h,\phi^h) &&= \big( a \left( \!\tfrac{ h^{P-1}}{(P-1)!}\!\right)\!\tfrac{\partial^P}{\partial x^P} w^h , \tau  \left( \!\tfrac{ h^{P-1}}{(P-1)!}\!\right) \tfrac{\partial^{P-1}}{\partial x^{P-1}} \mathcal{L}{\phi^h} \big)_K  \\
& \tilde{B}_\text{vol}( w^h,\phi^h)&&= \big( a \,\tfrac{\partial}{\partial x} w^h, \tau_{\text{eff}}  \mathcal{L}{\phi^h} \big)_K \\
& B_\text{bdy}( w^h,\phi^h) &&=  \big\langle  a \left( \!\tfrac{ h^{P-1}}{(P-1)!}\!\right)\! \tfrac{\partial^P}{\partial x^P} w^h ,\gamma \,\phi^h \big\rangle_{\partial K \cap\partial \Omega_D^+ } \\
& \tilde{B}_\text{bdy}( w^h,\phi^h) &&= \big\langle a \,\tfrac{\partial}{\partial x} w^h ,\gamma_{\text{eff}} \, \phi^h \big\rangle_{\partial K \cap\partial \Omega_D^+ } 
\end{alignat}
\end{subequations}
The impact of these different bilinear forms may be quantified by considering their norms:
\begin{align}
\norm{B} := \sup\limits_{\tfrac{\partial}{\partial x}w^h\neq 0, \tfrac{\partial}{\partial x}\phi^h\neq 0}\frac{| B(w^h,\phi^h) |}{\norm{\tfrac{\partial}{\partial x}w^h}_{L^2(K)}\norm{\tfrac{\partial}{\partial x}\phi^h}_{L^2(K)}}\,,
\end{align}
where we have chosen to define the norm of the bilinear forms with respect to the $H^1$-seminorm of its arguments, as this seminorm is one of the terms in the optimality condition induced by the Nitsche projector according to \cref{NitscheMinimization}.

We choose $\tau_{\text{eff}}$ and $\gamma_{\text{eff}}$ such that the impact of these bilinear forms equal: $\norm{B_\text{vol}}=\norm{\tilde{B}_\text{vol}}$ and $\norm{B_\text{bdy}}=\norm{\tilde{B}_\text{bdy}}$. If we assume constant parameters in $K$, and the advective dominant case such that $\mathcal{L}{\phi^h}$ may be approximated by $a \tfrac{\partial}{\partial x} \phi^h$, then we obtain:
\begin{subequations}
\begin{alignat}{6}
    &\tau_{\text{eff}} &&\approx \tau \left( \!\tfrac{ h^{P-1}}{(P-1)!}\!\right)^2 \frac{\norm{  \big(\tfrac{\partial^P}{\partial x^P} \,\cdot\, , \tfrac{\partial^{P}}{\partial x^{P}} \,\cdot \big)_K }}{ \norm{\big( \tfrac{\partial}{\partial x} \,\cdot\,, \tfrac{\partial}{\partial x} \,\cdot \big)_K } } 
    = \tau \left( \!\tfrac{ h^{P-1}}{(P-1)!}\!\right)^2  \sup\limits_{\tfrac{\partial}{\partial x}w^h\neq 0}
    \frac{\norm{\tfrac{\partial^P}{\partial x^P}w^h}^2_{L^2(K)}  }
         {\norm{\tfrac{\partial}{\partial x}w^h}^2_{L^2(K)}      } \,,\\
    &\gamma_{\text{eff}}&&\approx \gamma \left( \!\tfrac{ h^{P-1}}{(P-1)!}\!\right) \frac{ \norm{ \big\langle \tfrac{\partial^P}{\partial x^P} \,\cdot\, , \,\cdot \big\rangle_{F }   } }{ \norm{ \big\langle \tfrac{\partial}{\partial x} \,\cdot\, , \, \cdot \big\rangle_{F }   }} 
    = \gamma \left( \!\tfrac{ h^{P-1}}{(P-1)!}\!\right)  
    \tfrac{ \sup\limits_{\tfrac{\partial}{\partial x}w^h\neq 0} \left( |\tfrac{\partial^P}{\partial x^P}w^h|\big|_{F} \,\middle/\, \norm{\tfrac{\partial}{\partial x}w^h}_{L^2(K)} \right) }  
         { \sup\limits_{\tfrac{\partial}{\partial x}w^h\neq 0} \left( |\tfrac{\partial}{\partial x}w^h|    \big|_{F} \,\middle/\, \norm{\tfrac{\partial}{\partial x}w^h}_{L^2(K)} \right) } \,.
\end{alignat}\label{taueffnorms}%
\end{subequations}

The inverse inequalities in these expressions are computable by hand \cite{harari1992c}. For linear, quadratic and cubic coarse-scale basis function, the relations between the parameters $\tau_{\text{eff}}$ and $\tau$, and $\gamma_{\text{eff}}$ and $\gamma$ become:
\begin{align}
    \tau_{\text{eff}},\, \gamma_{\text{eff}}  \approx \begin{cases} \begin{alignedat}{8}
    &\tau, &&\gamma  && \text{for }P=1\,, \\ 
    &12\,\tau, &&\sqrt{3}\,\gamma  && \text{for }P=2\,, \\ 
    &180\,\tau, \,\,&& 2\sqrt{5}\,\gamma  \qquad && \text{for }P=3\,.
    \end{alignedat}\end{cases} \label{taugameff}
\end{align}


\subsubsection{$\tau$-parameter approximation for $P\in \{1,2,3\}$}
\label{ssec:TauEst}

In literature, we find that the $\tau$-parameter that is used with higher-order basis functions is often the same as that for $P=1$ (i.e., obtained from the element local Green's function), sometimes with a $P$-dependent mesh size scaling. 
In this section, we propose an approximation of $\tau$ for linear, quadratic and cubic elements based on the actual fine-scale Green's functions. 
These are devised such that they limit to the exact expressions in the advection ($\tau_a$) or diffusion ($\tau_d$) dominated cases. Using the definition of $\tau$ from \cref{taudef}, together with the fine-scale Green's functions from \ref{sec:fine-scale green}, the following exact expressions for $\tau$ may be computed for $P=1,2$ and $3$ respectively \cite{Brooks1982,Hughes1995,Hughes2007}:
\begin{subequations}
\begin{alignat}{6}
    &\tau_1 &&= \frac{h}{2|\B{a}|} \left (\frac{2+Pe -(2-Pe)\exp(Pe)}{-Pe+Pe \exp(Pe)} \right)\, =: \frac{h}{2|\B{a}|} \xi_1(Pe) \label{taudefP1} \,, \\
    &\tau_2 &&= \frac{h}{72|\B{a}|} \left (\frac{ 12 + 6Pe + Pe^2 - (  12 - 6Pe + Pe^2 ) \exp(Pe)}{ - 2Pe - Pe^2 + ( 2Pe - Pe^2 )\exp(Pe)}\right) \,=:\frac{h}{2|\B{a}|} \xi_2(Pe) \label{taudefP2} \,, \\
    &\tau_3 &&= \frac{h}{1800|\B{a}|} \left (\frac{ 120+60 Pe+12 Pe^2+Pe^3 - (120-60 Pe+12 Pe^2-Pe^3) \exp(Pe)}{-12Pe - 6 Pe^2- Pe^3+(12 Pe-6 Pe^2 + Pe^3)\exp(Pe)}\right)  \nonumber \\
    &       &&\hspace{0.5cm}=:\frac{h}{2|\B{a}|} \xi_3(Pe) \,, \label{taudefP3} 
\end{alignat}\label{taudefP123}%
\end{subequations}
where $Pe =\frac{|\boldsymbol{a}|\,h}{\kappa}$ is the element P\'{e}clet number and $\xi$ is the \textit{upwind function}. From these equations we obtain the following advective and diffusive limits:
\begin{subequations}
\begin{alignat}{4}
&\tau_{1,a} := \lim\limits_{Pe\rightarrow\infty}\!\tau_1 = \frac{h}{2\,|\boldsymbol{a}|}\,,\qquad\qquad && \tau_{1,d} := \lim\limits_{Pe\rightarrow 0^+}\!\tau_1   = \frac{h^2}{12 \kappa}\,, \label{tau_ad_1d_p1} \\
&\qquad\tau_{2,a}  = \frac{h}{72\,|\boldsymbol{a}|} \,, \qquad\qquad&& \qquad\tau_{2,d}  =  \frac{h^2}{720 \kappa} \,,\label{tau_ad_1d_p2} \\
&\,\,\,\quad\tau_{3,a} = \frac{h}{1800\,|\boldsymbol{a}|} \,, \qquad\qquad&& \,\,\,\quad\tau_{3,d} =  \frac{h^2}{25200 \kappa} \,.\label{tau_ad_1d_p3} 
\end{alignat}\label{tau_ad_1d}%
\end{subequations}

The following approximation strategy for $\tau$ is used frequently in stabilized methods~\cite{Hughes1986,Tezduyar1991,Shakib1991,Tezduyar2000,Bazilevs2007}: 
\begin{align}
   \tau \approx  \frac{1}{\sqrt{\tau_a^{-2}+ \tau_d^{-2}}}  \,.
    \label{tauapproxdef}
\end{align}
To determine the effectiveness of the scaling for the various polynomial orders, we substitute \cref{tau_ad_1d} into (\ref{tauapproxdef}). In all cases, we can rewrite the expression to obtain the effective approximate upwind function. For example, for linear elements:
\begin{align}
    {\tau}_1 &\approx \frac{1}{\sqrt{\frac{4\,|\boldsymbol{a}|^2}{h^2}+  \frac{144 \kappa^2}{h^4}}} = \frac{h}{2|\B{a}|} \frac{1}{\sqrt{1+36 Pe^{-2}}} =: \frac{h}{2|\B{a}|} \,\tilde{\xi}_1(Pe) .
    \label{Xiapprox}
\end{align}
 \Cref{fig:xi} illustrates how the approximate upwind functions $\tilde{\xi}(Pe)$ relates to the exact upwind functions of \cref{taudefP123}. The figure shows that the approximation of $\tau$ according to \cref{tauapproxdef} has the correct asymptotic limits, and converges to these limits at the correct rates. We observe that this holds for each polynomial order.


\begin{figure}[!ht]
    \centering
    \subfloat[Upwind function, scaling of $\tau$.\label{fig:xi}]{\includegraphics[width=0.47\linewidth]{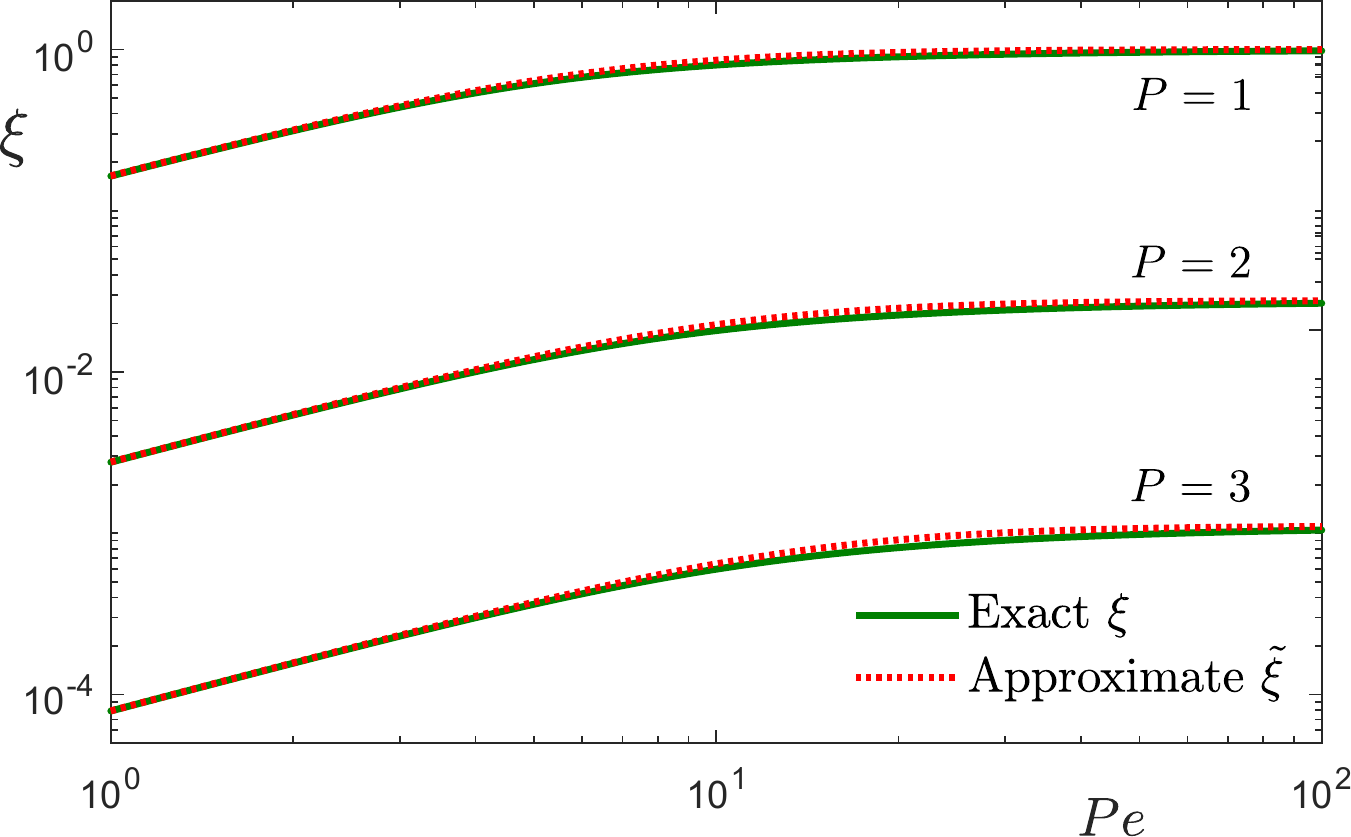}}\hspace{0.3cm}
    \subfloat[Boundary function, scaling of $\gamma$.\label{fig:eta}]{\includegraphics[width=0.47\linewidth]{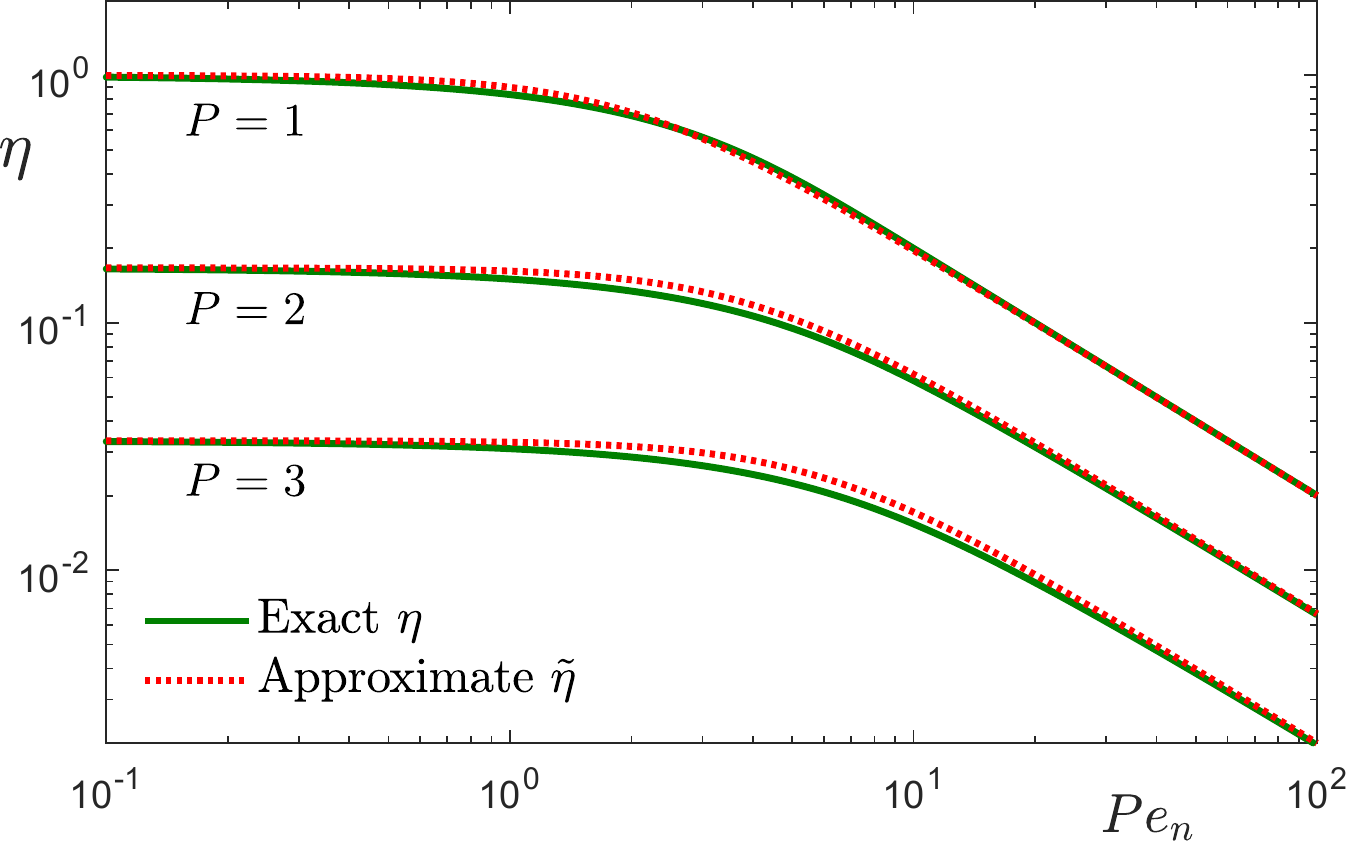}}
    \caption{Exact $\xi$ and $\eta$ functions and their approximations for the one-dimensional case. Showing the correct scaling behavior.}
    \label{fig:scaling}
\end{figure}

\subsubsection{$\gamma$-parameter approximation for $P\in \{1,2,3\}$}
\label{ssec:GamEst}
In a similar sense, we wish to construct an approximate $\gamma$, based on generic (spatial dimension independent) parameters, that share the asymptotic scaling behavior of the exact one-dimensional expression. The exact expressions of $\gamma$ may be computed from \cref{gamdef} as:
\begin{subequations}
\begin{alignat}{6}
    &\gamma_1 &&= \frac{h}{2}   \left(\frac{2 + 2 Pe_n -  2\exp( Pe_n )}{Pe_n -Pe_n\exp( Pe_n )}\right) =: \frac{h}{2} \,\eta_1 (Pe_n)\,, \label{gamdefP1}\\[0.2cm]
    &\gamma_2 &&= \frac{h}{12}\!\left(\dfrac{12  + 8Pe_n +2 Pe_n^2 - ( 12 - 4 Pe_n)\exp( Pe_n ) }{2Pe_n + Pe_n^2 - (2Pe_n - Pe_n^2 )\exp( Pe_n ) }\right) =: \frac{h}{2} \,\eta_2 (Pe_n)\,, \label{gamdefP2}  \\[0.2cm]  
    &\gamma_3 &&= \frac{h}{60}\!\left(\dfrac{120 + 72 Pe_n + 18 Pe_n^2+ 2 Pe_n^3   - (120 - 48 Pe_n + 6 Pe_n^2 )\exp( Pe_n ) }{12Pe_n  + 6 Pe_n^2 + Pe_n^3 - (12Pe_n - 6Pe_n^2 + Pe_n^3 )\exp( Pe_n )}\right) =: \frac{h}{2} \,\eta_3 (Pe_n) ,\label{gamdefP3}
\end{alignat}\label{gamdefP123}%
\end{subequations}
for $P=1,2$ and $3$ respectively. $Pe_n =\frac{\boldsymbol{a}\cdot \boldsymbol{n}\,h}{\kappa}$ is a boundary type element P\'{e}clet number. Since we only make use of $\gamma$ on the outflow boundary ($\B{a}\cdot \B{n}\geq 0$), we exclusively consider $Pe_n\geq 0$ in the following. The advective and diffusive limits of these expressions are:
\begin{subequations}
\begin{alignat}{4}
&\gamma_{1,a} := \lim\limits_{Pe_n\rightarrow\infty}\!\gamma_1 = 0 \,, \qquad\qquad && \gamma_{1,d} := \lim\limits_{Pe_n\rightarrow 0^+}\!\gamma_1   = \frac{h}{2}\,,\label{gamma_ad_1d_p1} \\
&\quad\gamma_{2,a}  = 0 \,, \qquad\qquad\qquad\quad&&\quad\gamma_{2,d}  =  \frac{h}{12}\,,\label{gamma_ad_1d_p2} \\
&\quad\gamma_{3,a} = 0 \,, \qquad\qquad\qquad\quad&&\quad\gamma_{3,d} =  \frac{h}{60}\,.\label{gamma_ad_1d_p3} 
\end{alignat}\label{gamma_ad_1d}%
\end{subequations}
Since these limits do not depend on $\B{a}$ nor on $\kappa$, the approximation strategy of \cref{tauapproxdef} is not viable. Instead, we deduce the following scaling of $\gamma$ based on its definition from \cref{gamdef}:
\begin{align}
    \gamma \propto \big| \mathbb{H} \big|\, \hat{\tau}, \label{gammaapproxtautilde}
\end{align}
where $|\mathbb{H}|$ represents the scaling induced by $\mathbb{H}$ as defined in \cref{H}, and $\hat{\tau}$ is a modified $\tau$-like quantity that takes into account the difference in domain of integration between $\gamma$ and $\tau$ in \cref{taugamdef}. $\mathbb{H}$ scales linearly with $\kappa$ and inversely with length, such that we may choose $|\mathbb{H}|\propto \sqrt{\kappa \tau_d^{-1}}$. By adopting a similar approximation for $\hat{\tau}$ as \cref{tauapproxdef}, we obtain:
\begin{align}
    \gamma \approx c_s \sqrt{\kappa \tau_d^{-1}}\sqrt{\frac{1}{c_1\tau_a^{-2}+ c_2\tau_d^{-2}}} = c_s \sqrt{ \frac{\kappa}{c_1 \tau_d \tau_a^{-2}+c_2 \tau_d^{-1}}}. \label{gamapproxdef}
\end{align}
$c_s$ takes into account the shape effect of the element as the measure of the relevant boundary versus the measure of the element interior: 
\begin{align}
   c_s  := \frac{ h |F| }{ |K| }  \label{cs}\,,
\end{align}
where the multiplication with $h$ ensures a mesh size independent scaling, and $h$ should be the representative element size that is used in the approximation of $\tau$ through $\tau_a$ and $\tau_d$. In this article we use the longest element edge. In the one-dimensional case $|F| = 1$ and $|K|=h$, such that $c_s = 1$. 

The coefficients $c_1$ and $c_2$ are introduced to capture the difference in scaling between $\tau$ and $\hat{\tau}$. They can be determined from \cref{gamapproxdef} by ensuring the correct limiting behavior of $\gamma$ in the one-dimensional case. For example, for linear elements, substituting the limits of \cref{tau_ad_1d_p1} results in:
\begin{align}
    {\gamma}_1 \approx \frac{h}{2}\,\sqrt{ \frac{1}{\frac{1}{12} c_1 Pe_n^2+3\,c_2 }} =:  \frac{h}{2} \, \tilde{\eta}_1(Pe_n) . \label{etaapprox}
\end{align}
By ensuring that \cref{etaapprox} has the same asymptotic limits as \cref{gamdefP123} and also has the same convergence rate towards zero, we obtain $c_1=3$ and $c_2=1/3$. The same strategy results in $c_1 = 1.25$ and $c_2 = 0.2$ for $P=2$ and $c_1 =7/9$ and $c_2 = 1/7$ for $P=3$. \Cref{fig:eta} shows the approximate and exact boundary functions and confirms that the approximation of $\gamma$ displays the correct asymptotic scaling behavior for all polynomial orders.

\subsection{Complete finite element formulation}
\label{ssec:summary}

With all the modeling terms included, the finite element formulation becomes: 
\begin{align}
&\text{Find } \phi^h \in \mathcal{W}^h \text{ s.t. } \forall\, w^h \in \WW^h:\nonumber\\
\begin{split}
&\hspace{5mm}B(w^h,\phi^h) = B_{\text{A}}(w^h,\phi^h) + B_{\text{D}}(w^h,\phi^h)  +  B_{\text{VMS,}\tilde{\Omega}}(w^h,\phi^h)  +  B_{\text{VMS,}\partial\Omega_D^+}(w^h,\phi^h)   \\
&\hspace{8mm}   = \big( w^h , f \big)_\Omega    \! + \big\langle w^h , g_N \big \rangle_{\partial\Omega_N} \!\! - \big\langle  \B{a} \cdot \B{n}\,  w^h ,\phi_D \big\rangle_{\partial\Omega_D^-}  \! - \big\langle  \kappa \pd_n w^h  , \phi_D \big\rangle_{\partial\Omega_D} \!\!+ \big\langle w^h ,  \kappa\beta \phi_D \big \rangle_{\partial\Omega_D} \hspace{1cm}
   \\
&\hspace{13mm} 
  + \big(\B{a}\cdot\nabla w^h \, \tau_{\text{eff}}, f \big)_{\Omega} +\big \langle\, \B{a}\cdot \!\nabla w^h \, \gamma_{\text{eff}} ,\phi_D\big\rangle_{\partial \Omega_D^+},\hspace{0.7cm}
 \end{split}\raisetag{1.7cm} \label{finalFEM}
\end{align}
where the advection and diffusion parts of the bilinear form are:
\begin{subequations}
\begin{alignat}{3}
  &B_{\text{A}}(w^h,\phi^h) = -\big(\B{a}\cdot \nabla w^h,\phi^h\big)_\Omega + \big \langle \B{a} \cdot \B{n}\, w^h,\phi^h\big \rangle_{\partial \Omega^+} ,\label{BA}\\
  \begin{split}
      &B_{\text{D}}(w^h,\phi^h) = \big(\kappa \nabla w^h, \nabla \phi^h\big)_\Omega 
      - \big \langle \kappa w^h, \partial_n \phi^h \big \rangle_{\partial \Omega_D} \!\! -\big \langle \kappa \partial_n w^h, \phi^h \big \rangle_{\partial \Omega_D} \!\!+ \big \langle \kappa \beta w^h, \phi^h \big \rangle_{\partial \Omega_D},
  \end{split} \label{BD}
\end{alignat}
\end{subequations}
and where the two variational multiscale components are:
\begin{subequations}
\begin{alignat}{3}
  &B_{\text{VMS,}\tilde{\Omega}}(w^h,\phi^h) =  \big( \B{a} \cdot \!\nabla w^h\, \tau_{\text{eff}}, \B{a}\cdot \!\nabla \phi^h- \nabla\cdot \kappa\nabla \phi^h \big)_{\tilde{\Omega}} ,\label{BVMS1} \hspace{5cm}\\
  &B_{\text{VMS,}\partial\Omega_D^+}(w^h,\phi^h) = \big \langle \B{a}\cdot \!\nabla w^h \, \gamma_{\text{eff}} ,\phi^h\big\rangle_{\partial \Omega_D^+}, \label{BVMS2}
\end{alignat}
\end{subequations}
with $\tilde{\Omega}$ the sum of open element domains.

Expressions for the parameters $\tau_\text{eff}$ and $\gamma_\text{eff}$ are collected in \Cref{tab:cs}. These expressions take into account all the considerations discussed in \cref{ssec:TauEst,ssec:GamEst,ssec:OpEst}. As the table shows, we have formulated all model parameters such that they depend exclusively on $\tau_{1,a}$ and $\tau_{1,d}$, i.e., those relating to linear elements. The exact approach for computing these limiting values remains flexible. For example, one could incorporate the Jacobian of the element mapping \cite{Bazilevs2007,Bazilevs2007weak_b}, use element local length-scales and P\'eclet numbers \cite{Tezduyar2000}, or use the analytical expressions of \cref{tau_ad_1d_p1} and the element diameter. 

\begin{table}[bht]
\caption{Overview of $\tau_{\text{eff}}$ and $\gamma_{\text{eff}}$ expressions for different elements and polynomial degrees.}
\label{tab:cs}
\centering
\begin{tabular}{lcccc} \hline\hline\\[-0.3cm]
        & $P=1$                               & \multicolumn{2}{c}{$P=2$  }                            & $P=3$   \\[0.1cm] \hline
        &                                     &                         &           &         \\[-0.2cm]
$\tau_{\text{eff}}$ :  & $\sqrt{ \dfrac{1}{  \tau_{1,a}^{-2}+  \tau_{1,d}^{-2}  } }$  & \multicolumn{2}{c}{$\sqrt{ \dfrac{1}{ 9 \,\tau_{1,a}^{-2}+  25 \,\tau_{1,d}^{-2}  } }$ } & $\sqrt{ \dfrac{1}{ 25 \,\tau_{1,a}^{-2}+ \tfrac{1225}{9}\, \tau_{1,d}^{-2}  } }$ \\[0.6cm]
$\gamma_{\text{eff}}$ :  & $c_s \sqrt{ \dfrac{\kappa}{3 \tau_{_{1,d}} \tau_{1,a}^{-2}+\frac{1}{3} \tau_{1,d}^{-1}}}$  & \multicolumn{2}{c}{$c_s \sqrt{ \dfrac{\kappa}{9 \tau_{_{1,d}} \tau_{1,a}^{-2}+ 4 \tau_{1,d}^{-1}}}$} & $c_s \sqrt{ \dfrac{\kappa}{15\tau_{_{1,d}} \tau_{1,a}^{-2}+15 \tau_{1,d}^{-1}}}$  \\[0.8cm]\hline\hline
         &   &   &   &   \\[-0.2cm]
\raisebox{1cm}{Element: } & \raisebox{0.8cm}{\includegraphics[width=0.12\textwidth]{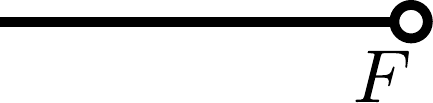}}&\hspace{-0.1cm}\includegraphics[width=0.11\textwidth]{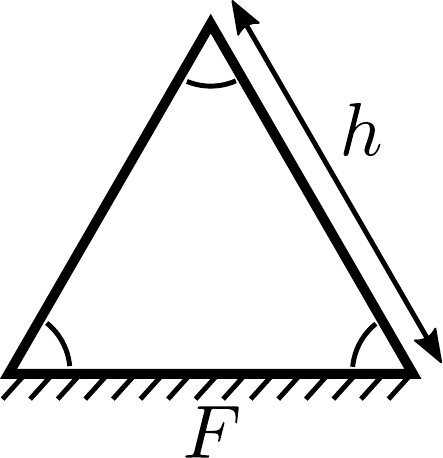}\hspace{0.1cm} &\hspace{0.4cm}\includegraphics[width=0.13\textwidth]{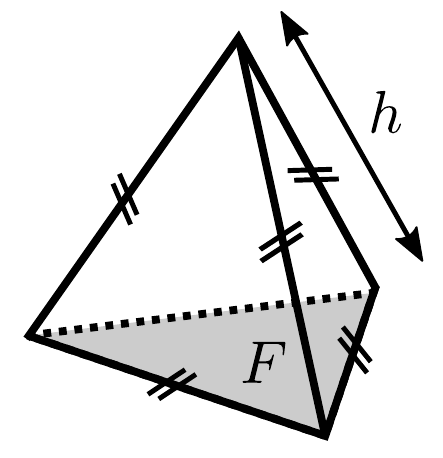} \hspace{-0.4cm} & \raisebox{0.2cm}{\includegraphics[width=0.13\textwidth]{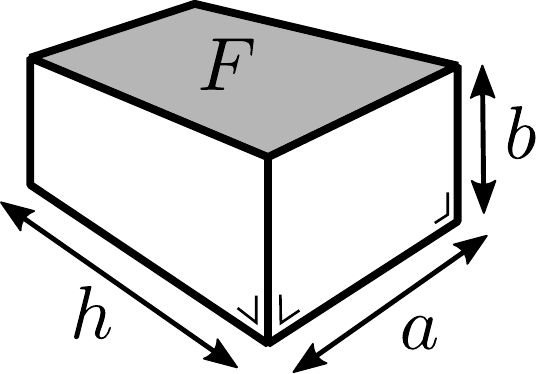}}\\
$c_s$ (\cref{cs}): \hspace{-2cm} &         1                         &         \hspace{-0.5cm}  $\frac{4}{\sqrt{3}}$                              &      \hspace{0.4cm}  $2\sqrt{\frac{2}{3}}$                 &           $h/b$                      \\[0.3cm] \hline\hline
\end{tabular}
\end{table}

\section{Analysis of existence and uniqueness}\label{sec:coerc}

As touched upon in \textit{Remark 4} in \cref{ssec:H10model}, the uniqueness of the fine-scale inversion is not a trivial issue. 
Additionally, the required simplification and modeling steps discussed in \cref{ssec:Generalization} may raise further questions regarding the existence of the approximate coarse-scale solution. In a finite-dimensional functional setting, existence and uniqueness follow directly from the coercivity of the bilinear form, which we analyze in this section.

In the analysis, we assume that $\B{a}$ and $\kappa$ are constant in $\Omega$ and that the grid is (quasi) uniform, such that also $\beta$, $\tau_a$ and $\tau_d$ can be chosen as global constants. The derivations can trivially be modified for non-uniform grids or non-constant $\B{a}$ and $\kappa$ and element-wise parameters $\beta$, $\tau_a$ and $\tau_d$. We further assume that the parameter $\tau_{\text{eff}}$ satisfies:
\begin{equation}
\tau_{\text{eff}} \leq C_d \tau_{1,d} \leq \inf\limits_{w^h\in\WW^h} \frac{1}{2 \kappa } \frac{\norm{ \nabla w^h}^2_{\Omega}}{\norm{\Delta w^h}^2_{\tilde{\Omega}} } \propto \frac{h^2}{\kappa}
\label{eq:tau_dCondition}
\end{equation}
The first inequality is satisfied due to the harmonic mean structure of $\tau_{\text{eff}}$, where the $C_d$'s follow directly from \cref{tab:cs} as 1, $\tfrac{1}{5}$ and $\tfrac{3}{35}$ for linear, quadratic and cubic basis functions respectively. On a one-dimensional mesh, this means that $C_d \tau_{1,d}=\frac{h^2}{12 \kappa}$, $\frac{h^2}{60 \kappa}$ and $\frac{h^2}{140 \kappa}$. In this one-dimensional case, the inverse estimate in \cref{eq:tau_dCondition} may be explicitly computed, resulting in maxima of $\infty$, $\frac{h^2}{24 \kappa}$ and $\frac{h^2}{120 \kappa}$ for polynomial orders of 1, 2 and 3. The condition of \cref{eq:tau_dCondition} is thus satisfied. We expect similar results in multiple spatial dimensions for meshes with reasonable quality.

Finally, we require that $\beta$ satisfies:
\begin{align}
 \beta \geq 4 \, (T_1 + c_s^2 T_2)  \propto\frac{1}{h},
\label{eq:betaCondition}
\end{align}
with:
\begin{subequations}\begin{align}
T_1 &= \sup\limits_{w^h\in\WW^h} \frac{\norm{\partial_n w^h}^2_{\partial\Omega_D}}{\norm{\nabla w^h}^2_\Omega}\propto\frac{1}{h}\label{eq:T_1}\,, \\
T_2 &= \sup\limits_{w^h\in\WW^h} \frac{\norm{\B{a}\cdot\nabla w^h}^2_{\partial\Omega_D^+}}{\norm{\B{a}\cdot \nabla w^h}^2_\Omega}\propto\frac{1}{h}\,.\label{eq:T_2}
\end{align}\end{subequations}

We carry out the coercivity proof for the different components of the bilinear forms separately, and then look at the formulation as a whole. We start with determining a relation between the model parameters $\gamma_{\text{eff}}$ and $\tau_{\text{eff}}$.

\begin{lem}\label{lem:ineq} 
For $P=1$, $2$ or $3$, the expressions from \cref{tab:cs} are such that $\gamma_{\text{eff}}$ is bounded by $\tau_{\text{eff}}$ according to:
\begin{align}
 \gamma_{\text{eff}}^2 \leq 3 c_s^2\, \kappa \tau_{\text{eff}}. \label{lem1}
\end{align}
\end{lem}
\begin{proof}
We write the $\gamma_{\text{eff}}$ and $\tau_{\text{eff}}$ expressions from \cref{tab:cs} in the following general form:
\begin{align}
  &  \tau_{\text{eff}} = \sqrt{ \frac{1}{(C_a \tau_{1,a})^{-2} + (C_d \tau_{1,d})^{-2} } }\\
  &  \gamma_{\text{eff}}  = c_s \sqrt{\frac{\kappa}{\tau_{1,d}}}\sqrt{ \frac{1}{(C_1 \tau_{1,a})^{-2} + (C_2 \tau_{1,d})^{-2} } }
\end{align}
Dividing $\gamma_{\text{eff}}$ by $\tau_{\text{eff}}$ and squaring gives:
\begin{align}
    \left( \frac{\gamma_{\text{eff}}}{\tau_{\text{eff}}} \right)^2 = c_s^2 \frac{\kappa}{\tau_{1,d}} \frac{(C_a \tau_{1,a})^{-2} + (C_d \tau_{1,d})^{-2} }{(C_1 \tau_{1,a})^{-2} + (C_2 \tau_{1,d})^{-2} }
\end{align}
After multiplying both sides by $C_d \tau_{1,d}$ and using $ C_d \tau_{1,d}\geq \tau$, we obtain:
\begin{align}
    C_d \tau_{1,d} \left( \frac{\gamma_{\text{eff}}}{\tau_{\text{eff}}} \right)^2 = c_s^2 \kappa \frac{C_d C_a^{-2} \tau_{1,a}^{-2} + C_d^{-1} \tau_{1,d}^{-2} }{(C_1 \tau_{1,a})^{-2} + (C_2 \tau_{1,d})^{-2} } \geq \tau_{\text{eff}}\left( \frac{\gamma_{\text{eff}}}{\tau_{\text{eff}}} \right)^2 = \frac{\gamma_{\text{eff}}^2}{\tau_{\text{eff}}}
\end{align}
The fraction may be bound from above as:
\begin{align}
    \frac{\gamma_{\text{eff}}^2}{\tau_{\text{eff}}} \leq c_s^2 \kappa \frac{\max( \tfrac{C_d}{C_a^2} , \tfrac{1}{C_d} )  ( \tau_{1,a}^{-2} + \tau_{1,d}^{-2}) }{\min(\tfrac{1}{C_1^2},\tfrac{1}{C_2^2} ) ( \tau_{1,a}^{-2} + \tau_{1,d}^{-2} ) }  = c_s^2 \kappa \frac{\max( \tfrac{C_d}{C_a^2} , \tfrac{1}{C_d} )   }{\min(\tfrac{1}{C_1^2},\tfrac{1}{C_2^2} ) } = \begin{cases}
     3\, c_s^2 \kappa \quad&\text{for }P=1\,,\\
     \tfrac{5}{4}\,c_s^2 \kappa \quad&\text{for }P=2\,,\\
     \tfrac{7}{9}\,c_s^2 \kappa \quad&\text{for }P=3\,.
    \end{cases}
\end{align}
\Cref{lem1} follows from the maximum of the three cases.
\end{proof}

\begin{lem}\label{lem:Adv} 
The bilinear form in \cref{BA} satisfies the following coercivity result:
\begin{align}
  B_\text{A}(w^h,w^h) \geq   \onehalf\norm{ \sqrt{|\B{a}\!\cdot\! \B{n}|}w^h}^2_{\partial \Omega} \quad \forall\,w^h \in \WW^h, \label{lem2}
\end{align}
where the norms are $L^2$ norms on the indicated domains. 
\end{lem}
\begin{proof}
Direct substitution of $\phi^h=w^h$ into \cref{BA} results in:
\begin{align}
    B_{\text{A}}(w^h,w^h) = -\big(\B{a}\cdot \nabla w^h,w^h\big)_\Omega + \big \langle \B{a}\cdot \B{n}\, w^h,w^h\big \rangle_{\partial \Omega^+}   \label{P1}.
\end{align}
Making use of the property $\nabla\cdot \B{a}=0$, the first term may be rewritten as follows:
\begin{align}
     -\int\limits_{\Omega} \nabla \cdot (\onehalf\B{a}\, (w^h)^2 )= - \onehalf \int\limits_{\partial\Omega} \B{a}\!\cdot\!\B{n}\, (w^h)^2 = \onehalf \int\limits_{\partial\Omega^-} |\B{a}\!\cdot\! \B{n}|\, (w^h)^2  - \onehalf \int\limits_{\partial\Omega^+} |\B{a}\!\cdot\! \B{n}| \, (w^h)^2 .
\end{align}
Substitution into \cref{P1} completes the proof.
\end{proof}

\begin{lem}\label{lem:Dif} 
Under the condition posed by \cref{eq:tau_dCondition}, the bilinear form in \cref{BD} satisfies the following coercivity result:
\begin{align}
  B_\text{D}(w^h,w^h) \geq \onehalf \kappa \, \norm{\nabla w^h}_{\Omega}^2 + (\beta-2T_1)\kappa \norm{w^h}^2_{\partial \Omega_D} \quad \forall\, w^h \in \WW^h, \label{lem3}
\end{align}
where the norms are $L^2$ norms on the indicated domains, and $T_1$ is given in \cref{eq:T_1}. 
\end{lem}
\begin{proof}
Direct substitution of $\phi^h=w^h$ into \cref{BD} results in:
\begin{align}
B_{\text{D}}(w^h,w^h) =  
\kappa \norm{\nabla w^h}_{\Omega}^2 
 -2\,\big \langle \kappa\,\partial_n  w^h , w^h \big \rangle_{\partial \Omega_D} 
 +\beta \kappa\norm{w^h}^2_{\partial \Omega_D} \label{lem3subs}.
\end{align}
By using Young's inequality to bound the nonsymmetric term, we obtain:
\begin{align}
-2\,\big \langle \kappa\, \partial_n  w^h , w^h \big \rangle_{\partial \Omega_D} 
& \geq 
- \epsilon\kappa \norm{ \, \partial_n  w^h }^2_{\partial \Omega_D} 
- \frac{\kappa}{\epsilon} \norm{w^h }^2_{\partial \Omega_D}, \nonumber \\
&  \geq - \epsilon T_1 \kappa \norm{ \nabla w^h }^2_{\Omega} 
- \frac{\kappa}{\epsilon} \norm{w^h }^2_{\partial \Omega_D},
\label{L2Nitsche}
\end{align}
where $T_1$ is defined in \cref{eq:T_1}. Choosing the parameter from Young's inequality as $\epsilon=1/(2T_1)$ completes the proof.
\end{proof}

\begin{lem}\label{lem:VMS1} 
Under the condition posed by \cref{eq:tau_dCondition}, the volumetric variational multiscale term in \cref{BVMS1} satisfies the following coercivity result:
\begin{align}
   B_{\text{VMS,}\tilde{\Omega}}(w^h,w^h) \geq \onehalf \tau_{\text{eff}} \, \norm{\B{a}\cdot\nabla w^h}_{\Omega}^2 - \onequart \kappa \, \norm{\nabla w^h}_{\Omega}^2\quad \forall\, w^h \in \WW^h, \label{eq:lemVMS1}
\end{align}
where the norms are $L^2$ norms on the indicated domains. 
\end{lem}
\begin{proof}
Direct substitution of $\phi^h=w^h$ results in:
\begin{align}
B_{\text{VMS,}\tilde{\Omega}}(w^h,w^h) =  \tau_{\text{eff}}\norm{\B{a} \cdot \!\nabla w^h}^2 +\big( \sqrt{\tau_{\text{eff}}}\, \B{a}\cdot \!\nabla w^h, \sqrt{\tau_{\text{eff}}}\, \kappa \Delta w^h \big)_{\tilde{\Omega}}.
\end{align}
With Young's inequality we bound the second term from below:
\begin{align}
B_{\text{VMS,}\tilde{\Omega}}(w^h,w^h) \geq  \tau_{\text{eff}}\norm{\B{a} \cdot \!\nabla w^h}^2 - \onehalf  \tau_{\text{eff}}\norm{\B{a} \cdot \!\nabla w^h}^2 - \onehalf \tau_{\text{eff}} \kappa ^2 \norm{\Delta w^h}^2 .
\end{align}
Using the assumed bound of $\tau_{\text{eff}}$ from \cref{eq:tau_dCondition} completes the proof.
\end{proof}

\begin{lem}\label{lem:VMS2} 
The boundary variational multiscale term in \cref{BVMS2} satisfies the following coercivity result:
\begin{align}
   B_{\text{VMS,}\partial\Omega_D^+}(w^h,\phi^h) \geq -\onequart \tau_{\text{eff}} \norm{\B{a}\cdot \!\nabla w^h}^2_\Omega -3\,c_s^2 T_2\kappa\norm{ w^h}^2_{\partial \Omega_D^+} \quad  \forall\, w^h \in \WW^h,  \label{eq:lemVMS2}
\end{align}
where the norms are $L^2$ norms on the indicated domains. 
\end{lem}
\begin{proof}

After substitution of $\phi^h=w^h$, we obtain:
\begin{align}
B_{\text{VMS,}\partial\Omega_D^+}(w^h,w^h) &= \big \langle \B{a}\cdot \!\nabla w^h , \gamma_{\text{eff}}  w^h\big \rangle_{\partial \Omega_D^+} \\
 &\geq 
-\onehalf \varepsilon \norm{\B{a}\cdot \!\nabla w^h}^2_{\partial \Omega_D^+}   
-\onehalf  \frac{\gamma_{\text{eff}}^2}{\varepsilon}\norm{ w^h}^2_{\partial \Omega_D^+} \nonumber  \\
& \geq
-\onehalf \varepsilon T_2\norm{\B{a}\cdot \!\nabla w^h}^2_\Omega
-\onehalf \frac{ \gamma_{\text{eff}}^2 }{\varepsilon}\norm{ w^h}^2_{\partial \Omega_D^+} . 
\end{align}
The first inequality follows from Young's inequality with parameter $\varepsilon$, and the second inequality as well as the parameter $T_2$ originate from the inverse estimate of \cref{eq:T_2}.
Choosing the parameter $\varepsilon=\tau_{\text{eff}}/(2T_2)$ and using the result of \textit{Lemma} \ref{lem:ineq} to relate $\tau_{\text{eff}}$ and $\gamma_{\text{eff}}$ completes the poof. 
\end{proof}

\begin{thm} 
The combined bilinear form of \cref{finalFEM} satisfies the following coercivity result:
\begin{align}
  B(w^h,w^h)  \geq  & \onequart \tau_{\text{eff}} \norm{\B{a}\cdot \!\nabla w^h}^2_{ \Omega} +\onehalf \norm{ \sqrt{|\B{a}\!\cdot\! \B{n}|} w^h}^2_{\partial \Omega}  +\onequart \kappa \norm{\nabla w^h}^2_{\Omega}  
   + \onequart \beta \kappa \norm{ w^h}^2_{\partial \Omega_D}\quad \forall\, w^h \in \WW^h, \label{thm1}
\end{align}\\[-0.5cm]
where the norms are $L^2$ norms on the indicated domains. 
\end{thm}
\begin{proof}
Direct substitution of $\phi^h=w^h$ in the bilinear form, and using the results of \textit{Lemmas} \ref{lem:Adv} to \ref{lem:VMS2}, results in:
\begin{align}
  B(w^h,w^h)  &=\,  B_{\text{A}}(w^h,w^h) + B_{\text{D}}(w^h,w^h)  + B_{\text{VMS,}\tilde{\Omega}}(w^h,w^h) + B_{\text{VMS,}\partial\Omega_D^+}(w^h,w^h)  \nonumber \\
  &\geq \, \onehalf\norm{ \sqrt{|\B{a}\!\cdot\! \B{n}|}w^h}^2_{\partial \Omega} +  \onehalf \kappa \, \norm{\nabla w^h}_{\Omega}^2 + (\beta-2T_1)\kappa \norm{w^h}^2_{\partial \Omega_D}  + \onehalf \tau_{\text{eff}} \, \norm{\B{a}\cdot\nabla w^h}_{\Omega}^2  \nonumber\\
  &\qquad - \onequart \kappa \, \norm{\nabla w^h}_{\Omega}^2  -\onequart \tau_{\text{eff}} \norm{\B{a}\cdot \!\nabla w^h}^2_\Omega -3\,c_s^2 T_2\kappa\norm{ w^h}^2_{\partial \Omega_D^+} \\
  &= \, \onequart \tau_{\text{eff}} \norm{\B{a}\cdot \!\nabla w^h}^2_{ \Omega} +\onehalf \norm{ \sqrt{|\B{a}\!\cdot\! \B{n}|} w^h}^2_{\partial \Omega}  +\onequart \kappa \norm{\nabla w^h}^2_{\Omega} \nonumber \\
  &\qquad + (\beta-2T_1-3\,c_s^2 T_2)\kappa \norm{w^h}^2_{\partial \Omega_D}   . \label{eq:thm1init}
\end{align}   
Using the assumption on $\beta$ from \cref{eq:betaCondition} completes the proof.
\end{proof}

\textbf{Remark 5:} Note that both Nitsche's method (\textit{Lemma} \ref{lem:Dif}) and the VMS method (\textit{Lemma} \ref{lem:VMS1}) rely on the first term in \cref{lem3subs} for their stability. As a result, the combined use of VMS and weakly enforced boundary conditions requires a larger $\beta$/$T_1$ ratio compared to the typical choice of $\beta = 2T_1$ for the standard Nitsche's method \cite{Embar2010}. 

\section{Numerical verification for a one-dimensional model problem}
\label{sec:numexpP1}

We present a number of numerical experiments to verify the derivation from \cref{ssec:InverseFSP,ssec:H10model,ssec:OneDModel}, and to investigate the accuracy improvement that may be achieved by using the new residual-based fine-scale model of sections \ref{ssec:Generalization} and \ref{ssec:summary}.

\subsection{Linear basis functions}
\label{ssec:num1D}


\Cref{fig:1DLin} shows the result for a simulation with $\B{a} = 0.8$, $\kappa = 0.02$ and $\beta=2/h$ on the domain $\Omega=[\,0,0.3\,]$, discretized with three linear elements. The solid green line shows the exact solution $\phi$. With the current discretization, the boundary layer falls completely within a single element. We obtain the exact coarse-scale solution, indicated with the black line, by projecting the exact solution onto the finite element mesh with the Nitsche projector from \Cref{ssec:NitscheVMS}. The blue dotted line is obtained by only using the classical VMS term, equivalent to $\gamma=0$, whereas the red dashed solution incorporates the exact augmented VMS model from \cref{HOexact} with the exact parameter definitions from \cref{taugamdef}. Finally, the line with the circular markers shows the result when the generalized model of sections \ref{ssec:Generalization} and \ref{ssec:summary}.
\begin{figure}[!b]
    \centering
    \subfloat[Coarse-scale solutions.]{\includegraphics[width=0.49\linewidth,trim={0.3cm 0.5cm 0.3cm 0.3cm},clip]{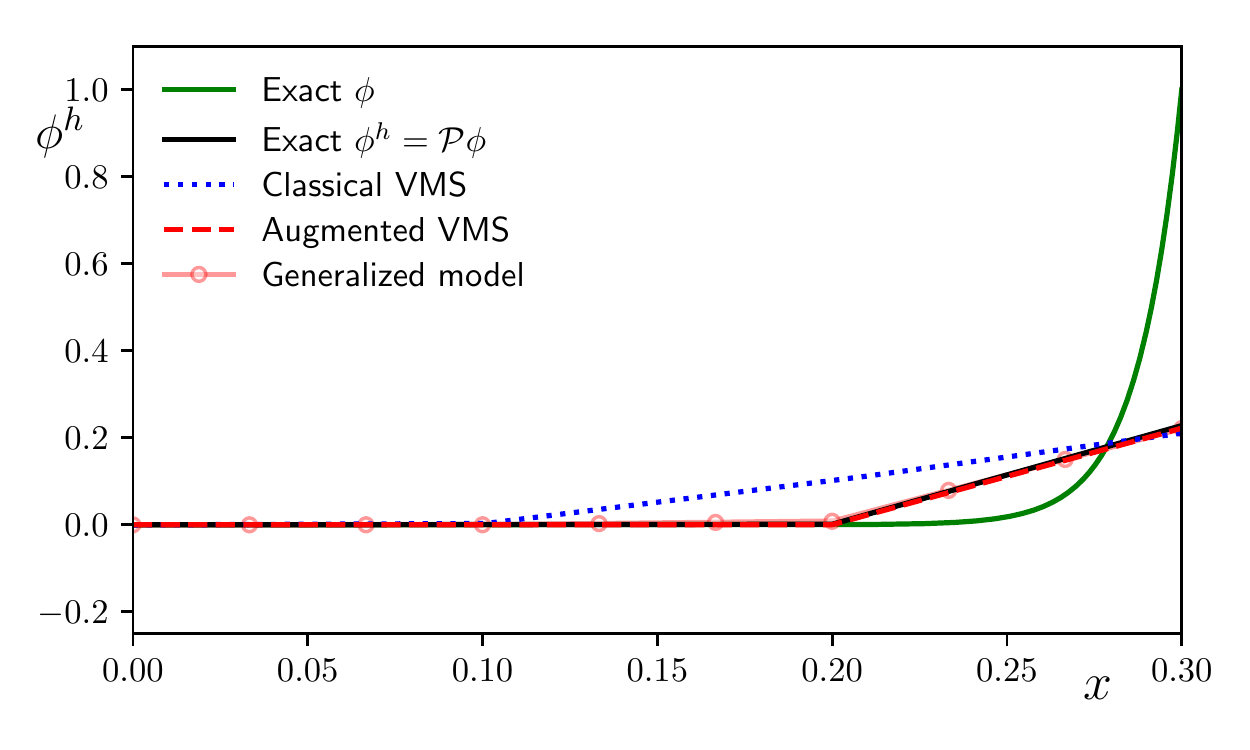} \label{fig:1DLinSol}}
    \subfloat[Fine-scale solutions (errors).]{\includegraphics[width=0.49\linewidth,trim={0.3cm 0.5cm 0.3cm 0.3cm},clip]{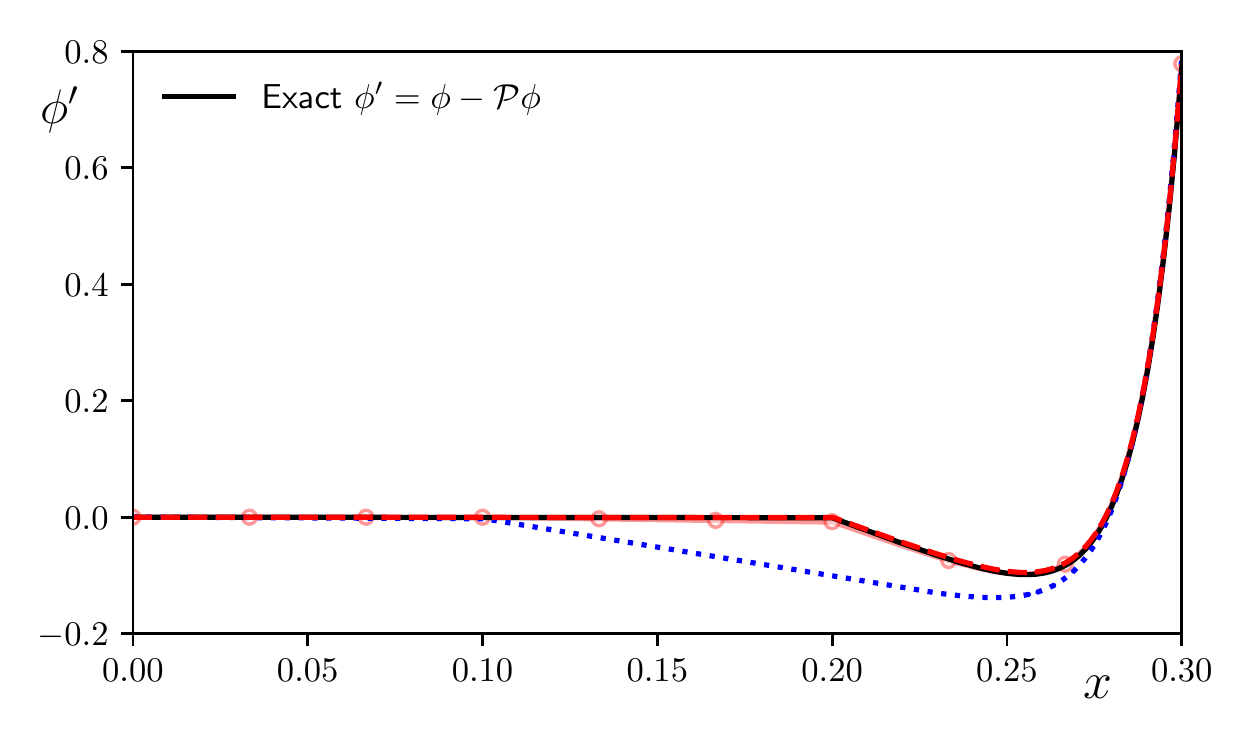} \label{fig:1DLinErr}}
    \caption{One-dimensional results using three linear elements.}
    \label{fig:1DLin}
\end{figure}

It is a celebrated fact that for this model problem the VMS method with strongly enforced boundary conditions results in nodally exact solutions \cite{Hughes1998}. The results of \Cref{fig:1DLinSol} show that this property is lost when the boundary conditions are enforced weakly, and that the under-resolved boundary layer affects the approximation on a large part of the domain. \Cref{fig:1DLinErr} shows that the magnitude of the fine-scale solution on the outflow boundary of the domain is considerable. 
Hence, the assumption of vanishing fine scales, which is critical in the derivation of the classical VMS model, is severely violated. 
In contrast, the augmented VMS model is exactly the Nitsche projection of the exact solution. As a result, the nodal exactness of the computational solution is retrieved and the adverse effect of the boundary layer is constrained to a single element. Due to the relative simplicity of this problem, solution obtained with the generalized model 
is only affected by the estimation of the model parameters. As it is nearly identical to the exact coarse-scale solution we conclude that, at least for this simple case, the estimation strategy is effective.

\textbf{Remark 6:} It is well known that for the current case the classical VMS term simplifies to a (consistent) diffusion term. Interestingly, in a similar sense the augmented term in the VMS formulation simplifies to a reduced diffusion in the symmetric part of Nitsche's formulation. In this context, the solution obtained by using the classical VMS model may be interpreted as excessively diffusive in the boundary layer, which is (consistently) counteracted by the augmented VMS term.

\subsection{Higher-order basis functions}


We use the same problem formulation but discretize with three higher-order elements. With quadratic basis functions and $\beta=3/h$ we obtain the solutions from \Cref{fig:1DP2Sol,fig:1DP2Err}, and with cubic basis function and $\beta=6/h$ we obtain the solutions from \Cref{fig:1DP3Sol,fig:1DP3Err}. We can largely draw the same conclusions as for the linear basis functions: the fine-scale solution 
deviates significantly from zero at the outflow boundary. 
As a result, the solution quality of the classical VMS model is spoiled. By using the augmented VMS model we obtain the Nitsche projection of the exact solution. This gives a nodally exact solution, where the adverse effect of the boundary layer is contained within the boundary element. These points are all compliant with the theory of \cref{ssec:NitscheVMS}.

\begin{figure}[b!]
    \centering
    \subfloat[Coarse-scale solutions, $P=2$.]{\includegraphics[width=0.49\linewidth,trim={0.3cm 0.5cm 0.3cm 0.3cm},clip]{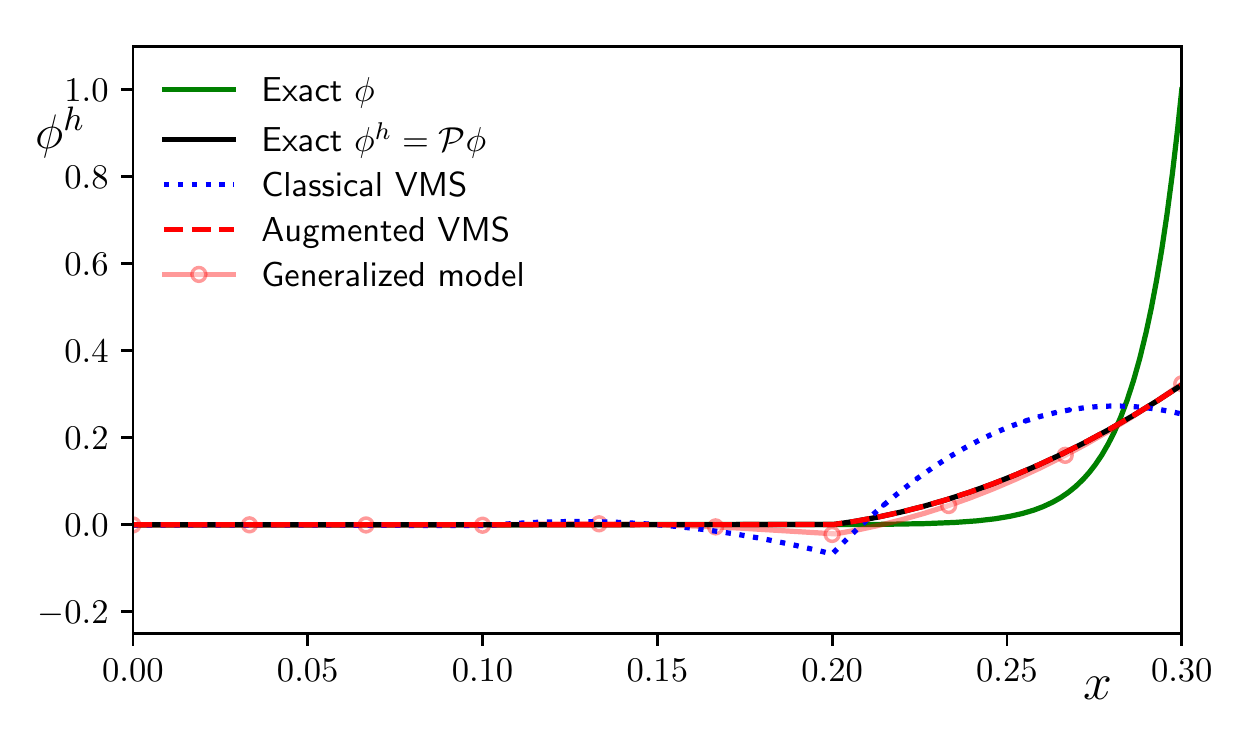} \label{fig:1DP2Sol}}
    \subfloat[Fine-scale solutions (errors), $P=2$.]{\includegraphics[width=0.49\linewidth,trim={0.3cm 0.5cm 0.3cm 0.3cm},clip]{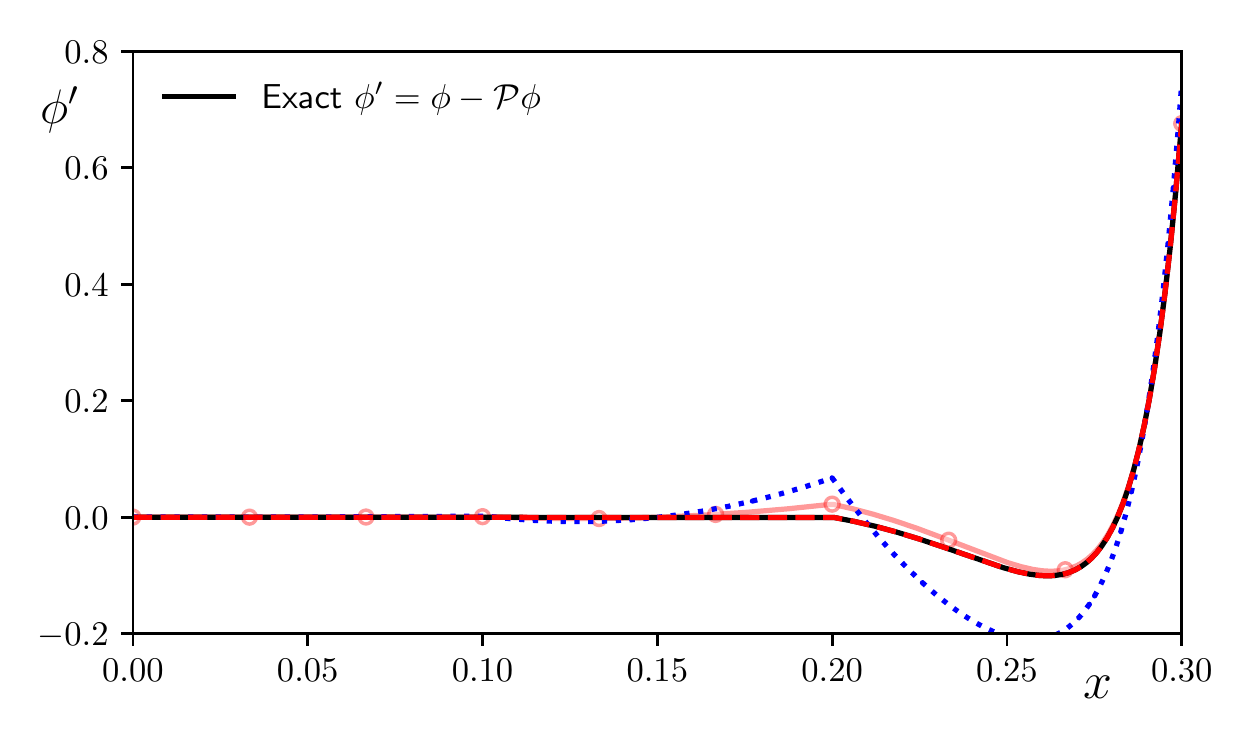} \label{fig:1DP2Err}}\\
    \subfloat[Coarse-scale solutions, $P=3$.]{\includegraphics[width=0.49\linewidth,trim={0.3cm 0.5cm 0.3cm 0.3cm},clip]{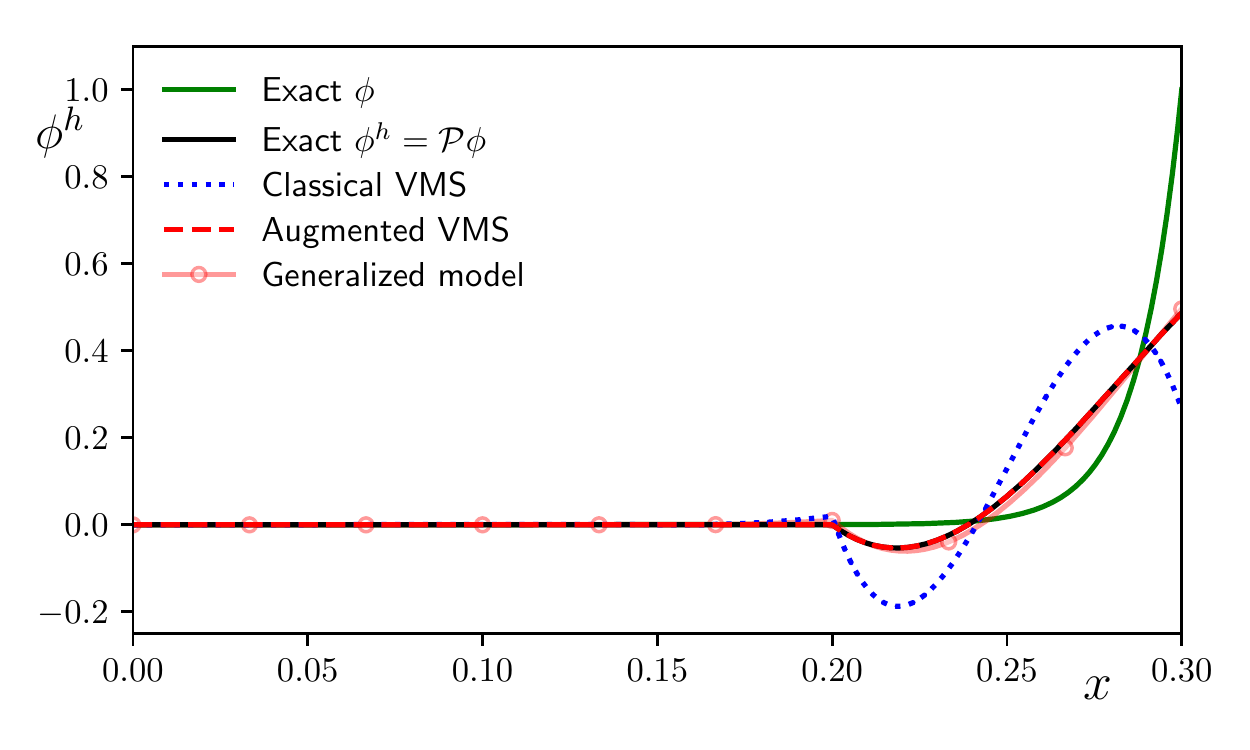} \label{fig:1DP3Sol}}
    \subfloat[Fine-scale solutions (errors), $P=3$.]{\includegraphics[width=0.49\linewidth,trim={0.3cm 0.5cm 0.3cm 0.3cm},clip]{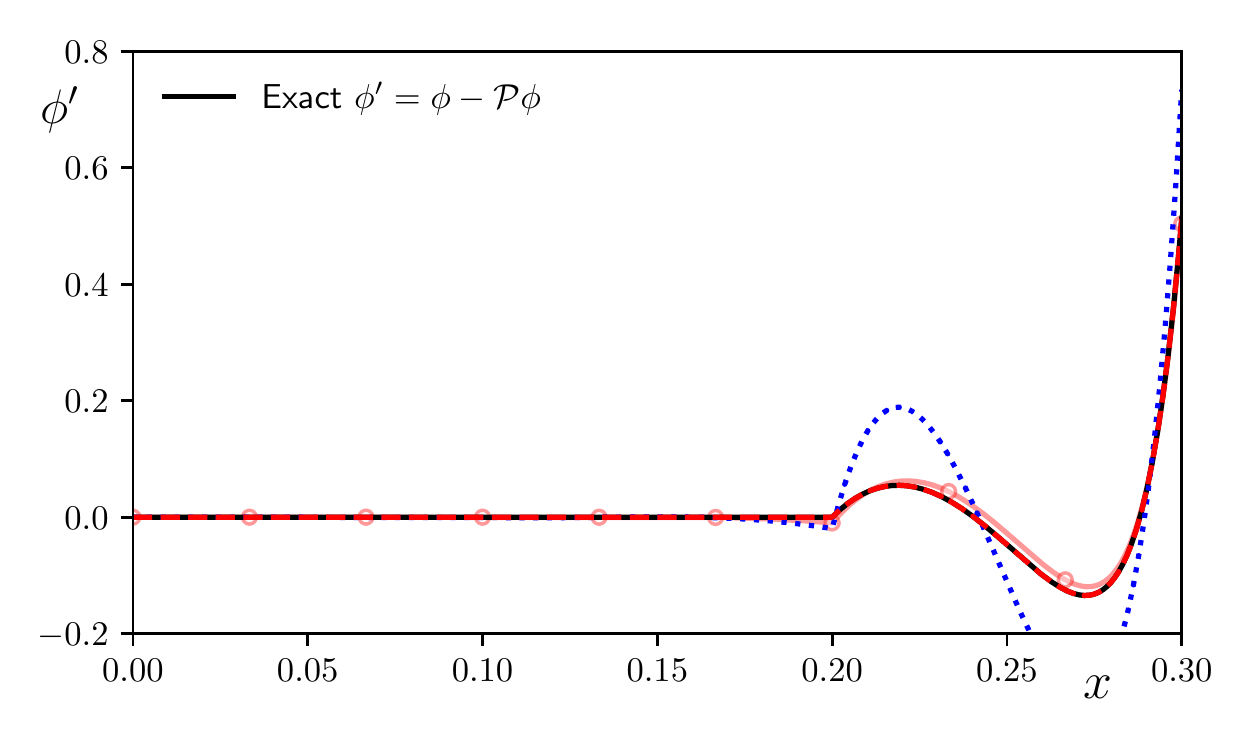} \label{fig:1DP3Err}}
    \caption{One-dimensional results using three quadratic or cubic elements.}
    \label{fig:1DHO}
\end{figure}

Additionally, we observe that the solution for the approximate augmented VMS model is very close to the exact coarse-scale solution. For the linear basis functions we concluded that the estimations of the model parameters are effective. We can now also conclude that the approximation of the differential operator described \cref{ssec:OpEst} is effective, at least for this one-dimensional case.\\

\section{Numerical experiments for a two-dimensional model problem}
\label{sec:numexpHO}

Next, we present numerical experiments for a two-dimensional domain. All the model approximations become important, and their effectiveness can be assessed. 

\subsection{Linear basis functions, high and low advective dominance}


We consider a model problem of a unit square with a circular hole of radius 0.24 in the center. Dirichlet conditions are enforced on all boundaries; $\phi_D = 0$ around the circular cut-out, and $\phi_D=x+y$ around the square. The advective field acts across the diagonal with a magnitude of $0.8$, the diffusivity is $\kappa=0.01$ or $\kappa=0.003$, and we use $\beta=10/h$. \Cref{fig:2Dmodel} schematically illustrates the model problem, and \Cref{fig:2Dexact} shows the solution for $\kappa=0.01$ obtained with a highly refined mesh. The solution features multiple boundary layers at various orientations.

The performance of the models can most clearly be assessed by investigating the resulting fine-scale solutions. These are shown in \Cref{fig:2DErrorLinLow,fig:2DErrorLinHigh} for $\kappa=0.01$ and $\kappa=0.003$ respectively. In \Cref{fig:2DErrorLinLowa,fig:2DErrorLinHigha} classical VMS stabilization is used (i.e. $\gamma=0$), and \Cref{fig:2DErrorLinLowb,fig:2DErrorLinHighb} show the results for the augmented model. Additionally, we project the overrefined solution onto the coarse-scale function space using the Nitsche projector, and show the resulting fine-scale solution in \Cref{fig:2DErrorLinLowc,fig:2DErrorLinHighc}. This represents the `exact' fine-scale solution.\\

\begin{figure}[H]
    \centering
    \subfloat[Problem specification.]{\includegraphics[width=0.3\linewidth]{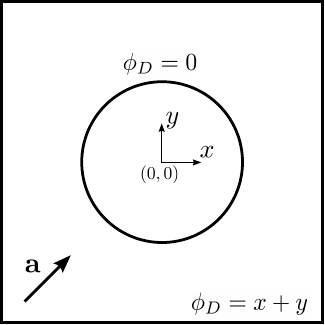}\label{fig:2Dmodel}} \hspace{1.3cm}
    \subfloat[Solution for $|\B{a}| =0.8$ and $\kappa=0.01$.]{\includegraphics[width=0.47\linewidth]{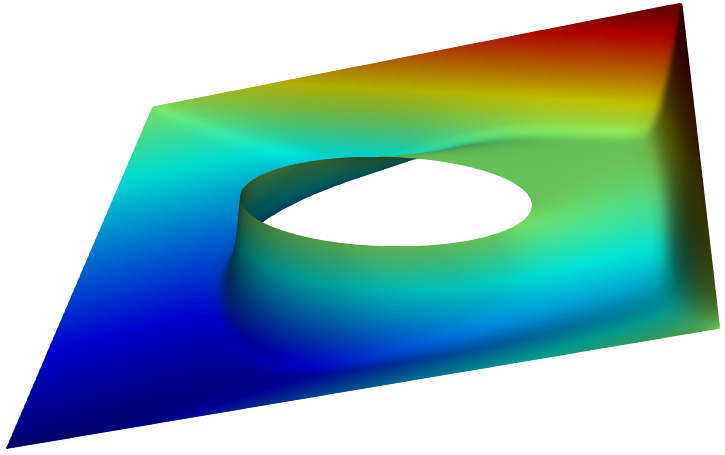} \label{fig:2Dexact}}
    \caption{Two-dimensional model problem for linear basis functions.}
    \label{fig:2Dproblem}
\end{figure}

\begin{figure}[H]
    \centering
    \vspace{0.5cm}
    \subfloat[Classical VMS model.] {\includegraphics[width=0.31\linewidth]{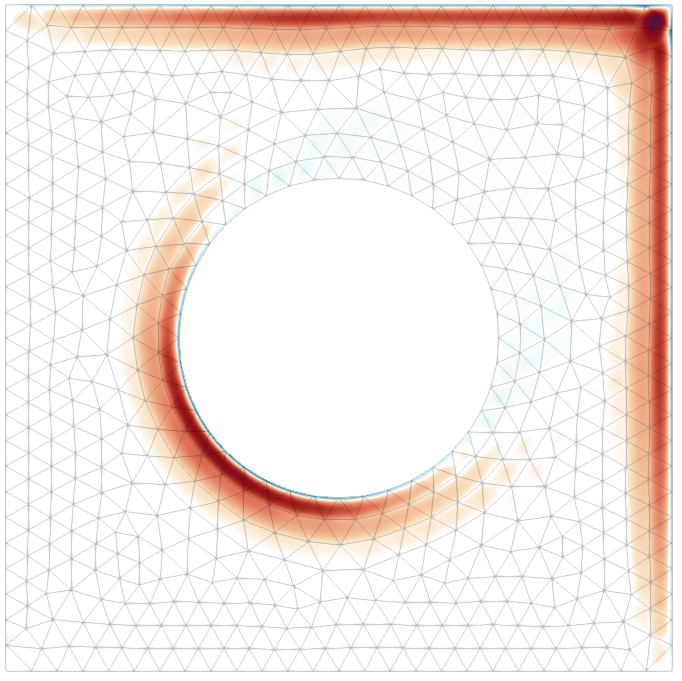}\label{fig:2DErrorLinLowa}} \hspace{0.2cm}
    \subfloat[Augmented VMS model.] {\includegraphics[width=0.31\linewidth]{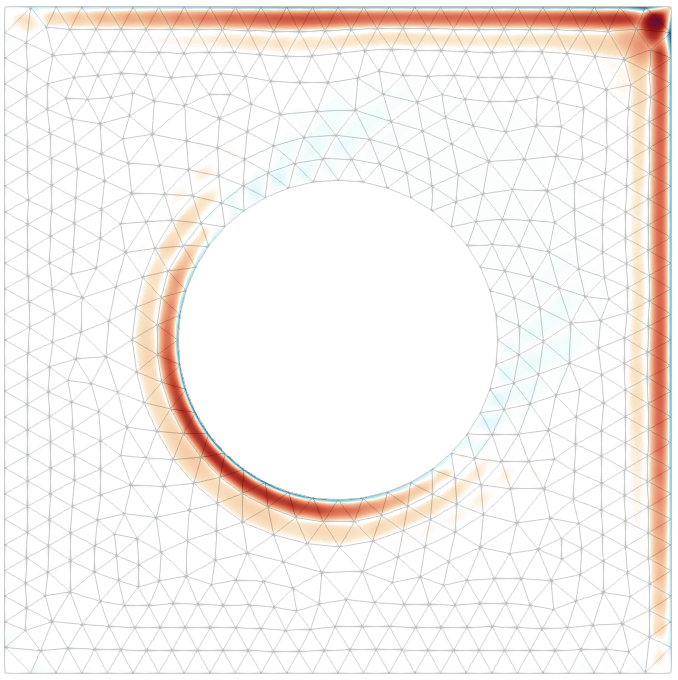} \label{fig:2DErrorLinLowb}}\hspace{0.2cm}
    \subfloat[Projected solution.] {\includegraphics[width=0.31\linewidth]{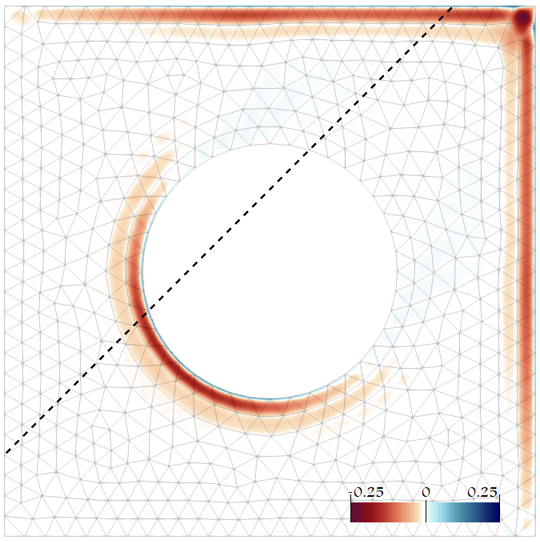} \label{fig:2DErrorLinLowc}}
    \caption{Fine-scale solutions $\phi-\phi^h$ (errors) for linear basis functions and $|\B{a}| =0.8$ and $\kappa=0.01$.}
    \label{fig:2DErrorLinLow}
\end{figure}
\begin{figure}[H]
    \centering
    \vspace{0.5cm}
    \subfloat[Classical VMS model.] {\includegraphics[width=0.31\linewidth]{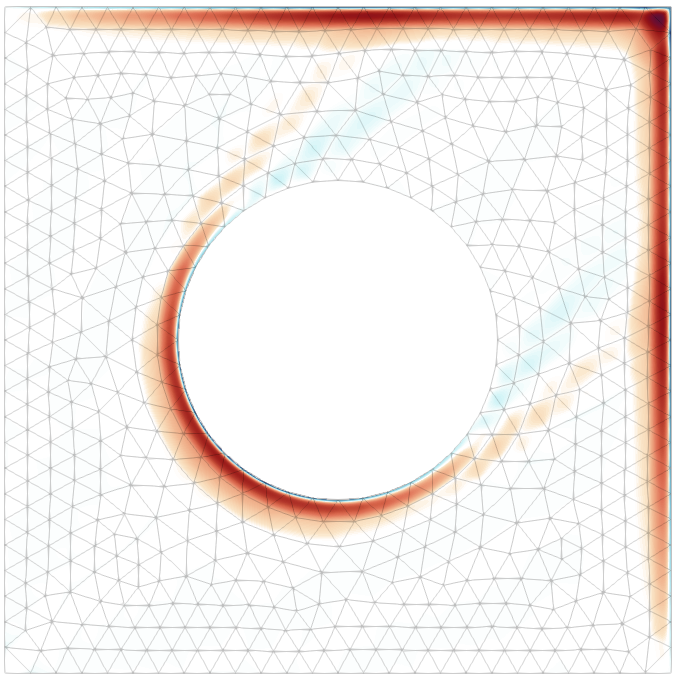}\label{fig:2DErrorLinHigha}} \hspace{0.2cm}
    \subfloat[Augmented VMS model.] {\includegraphics[width=0.31\linewidth]{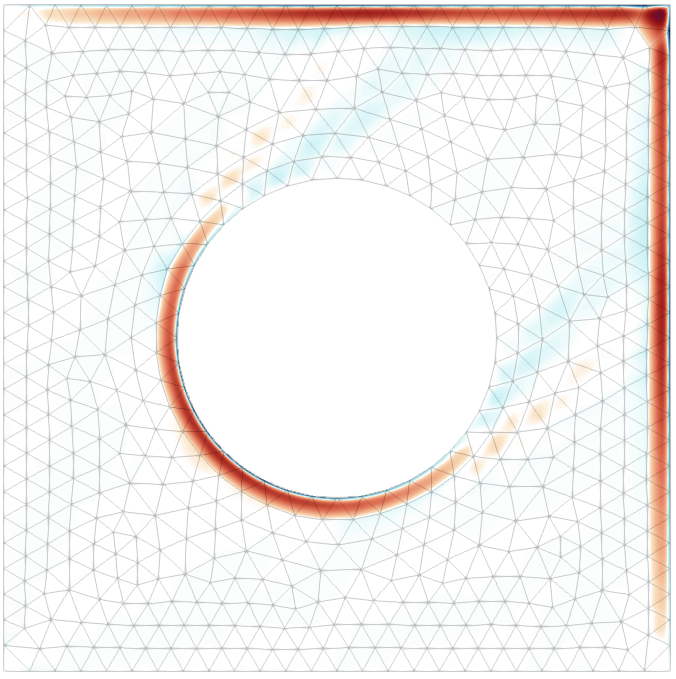} \label{fig:2DErrorLinHighb}}\hspace{0.2cm}
    \subfloat[Projected solution.] {\includegraphics[width=0.31\linewidth]{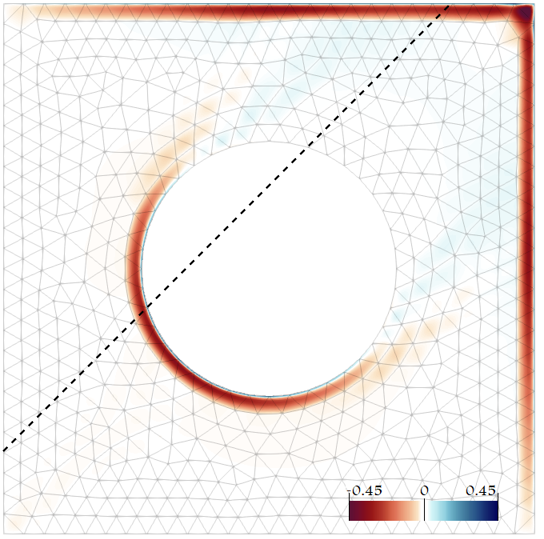} \label{fig:2DErrorLinHighc}}
    \caption{Fine-scale solutions $\phi-\phi^h$ (errors) for linear basis functions and $|\B{a}| =0.8$ and $\kappa=0.003$.}
    \label{fig:2DErrorLinHigh}
\end{figure}
\begin{figure}[H]
    \centering
    \subfloat[For $|\B{a}| =0.8$ and $\kappa=0.01$.]{\includegraphics[width=0.49\linewidth,trim={0.3cm 0.5cm 0.3cm 0.3cm},clip]{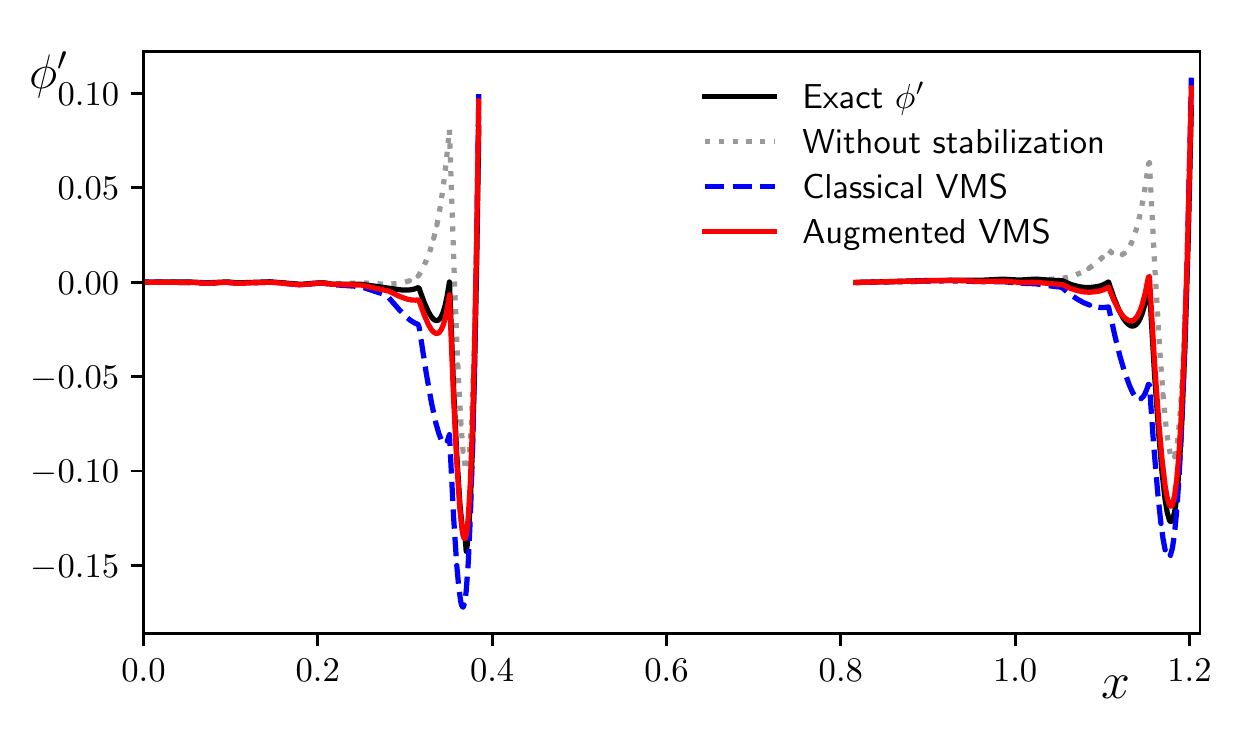} \label{fig:2DLinErrLow}}
    \subfloat[For $|\B{a}| =0.8$ and $\kappa=0.003$.]{\includegraphics[width=0.49\linewidth,trim={0.3cm 0.5cm 0.3cm 0.3cm},clip]{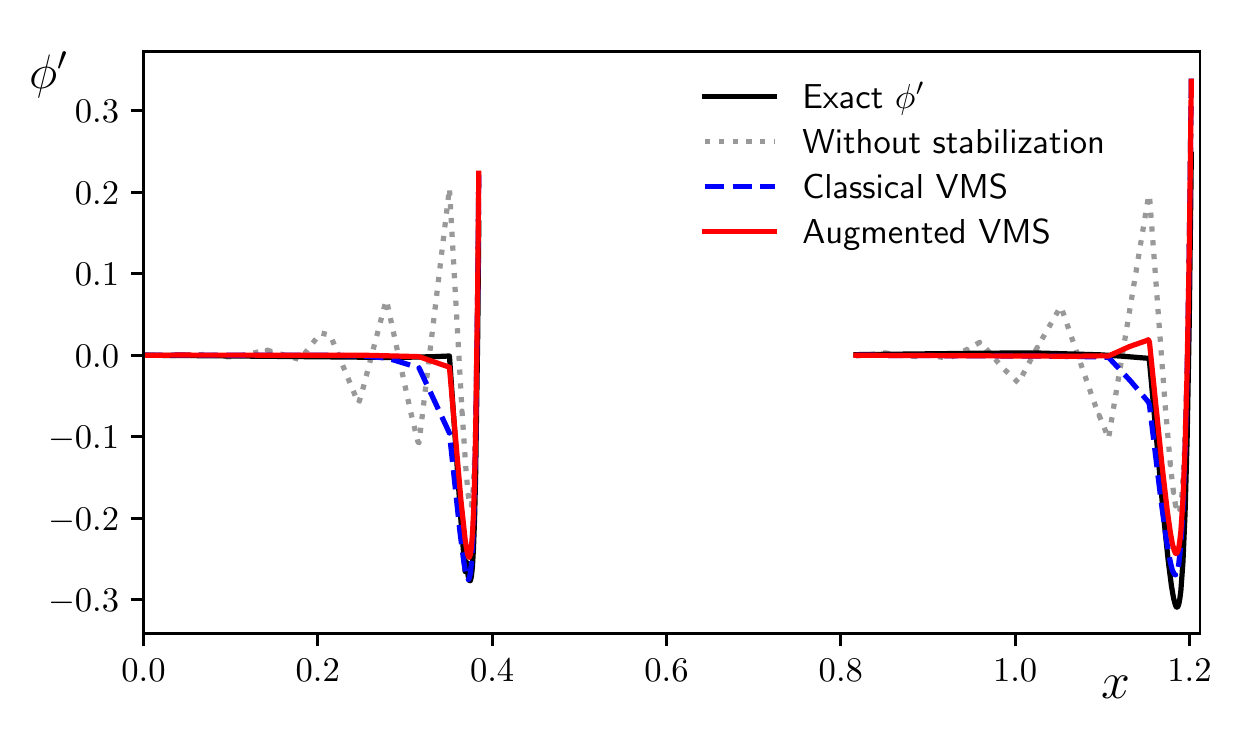} \label{fig:2DLinErrHigh}}
    \caption{Fine-scale solutions on the cut-planes from \cref{fig:2DErrorLinLowc,fig:2DErrorLinHighc}.}
    \label{fig:2DErr}
\end{figure}

The observations made for the one-dimensional case almost directly transfer to this two-dimensional problem. The Nitsche projector aims to constrain the impact of the high gradients to the boundary layer elements, without spoiling the results further into the domain. This is illustrated by large fine-scale solutions in only a single row of elements adjacent to the outflow boundary. When the classical VMS model is used, we observe a significant thickening of the range of nonzero fine scales; interpretable as excessive diffusion in the coarse-scale solution. When we add the additional modeling term this thickening is decreased, which leads to nearly the same solution quality as that obtained with the Nitsche projection. We observe these effects irrespective of the P\'eclet number. 

To further illustrate the significance of the change, we show all three solutions on a cut-plane in \Cref{fig:2DErr}. Note, in particular, the similarity of \Cref{fig:2DLinErrHigh} and the corresponding figure for the one-dimensional case (\Cref{fig:1DLinErr}). The fine-scale solution corresponding to a completely non-stabilized computation is also plotted to put the overall improvement of the solution quality into context.

Convergence in the $L^2$ or $H^1$ (semi)norms are not indicative of solution quality for the current case; $L^2$ projections of shocks lead to highly oscillatory solutions such that the non-stabilized solution often achieves the lowest $L^2$ error, and neither the classical nor the augmented VMS model aims to achieve optimality in the $H^1$ seminorm as the boundary conditions are not enforced strongly. Rather, the use of weakly enforced boundary conditions implies the optimality condition of \cref{optimalitycond}, satisfied by solutions that minimize \cref{NitscheMinimization}. 
The error of interest is thus the one with respect to the `optimal' solution; the exact coarse-scale solution obtained with the Nitsche projector. This error indicates the performance of the fine-scale model. This error is plotted for different mesh densities in \Cref{fig:2DL2P1low,fig:2DL2P1high}. 
\Cref{fig:2DErrorLinLow,fig:2DErrorLinHigh,fig:2DErr} correspond to the third data point in these convergence graphs. Both graphs show a considerable reduction in error when the augmented VMS model is used, which persists throughout mesh refinement. We observe that the ($L^2$) error reduction from classical to augmented VMS is often of the same magnitude, if not larger, than from non-stabilized to classical VMS. 

\begin{figure}[!htb]
    \centering
    \subfloat[For $|\B{a}| =0.8$ and $\kappa=0.01$.]{\includegraphics[width=0.4\linewidth,trim={0.4cm 0.0cm 0.4cm 0.4cm},clip]{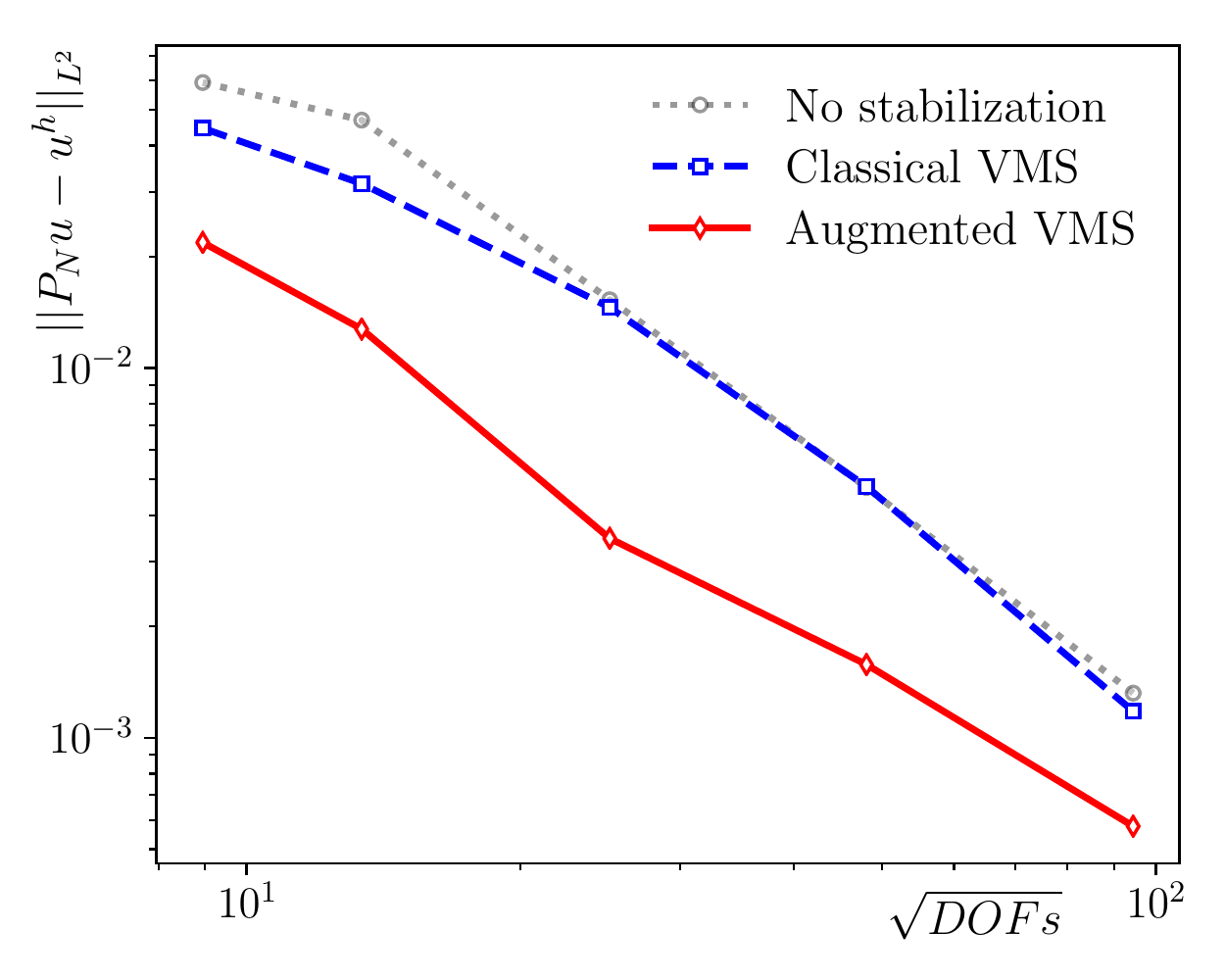}\label{fig:2DL2P1low}} \hspace{0.3cm}
    \subfloat[For $|\B{a}| =0.8$ and $\kappa=0.003$.] {\includegraphics[width=0.4\linewidth,trim={0.4cm 0.0cm 0.4cm 0.4cm},clip]{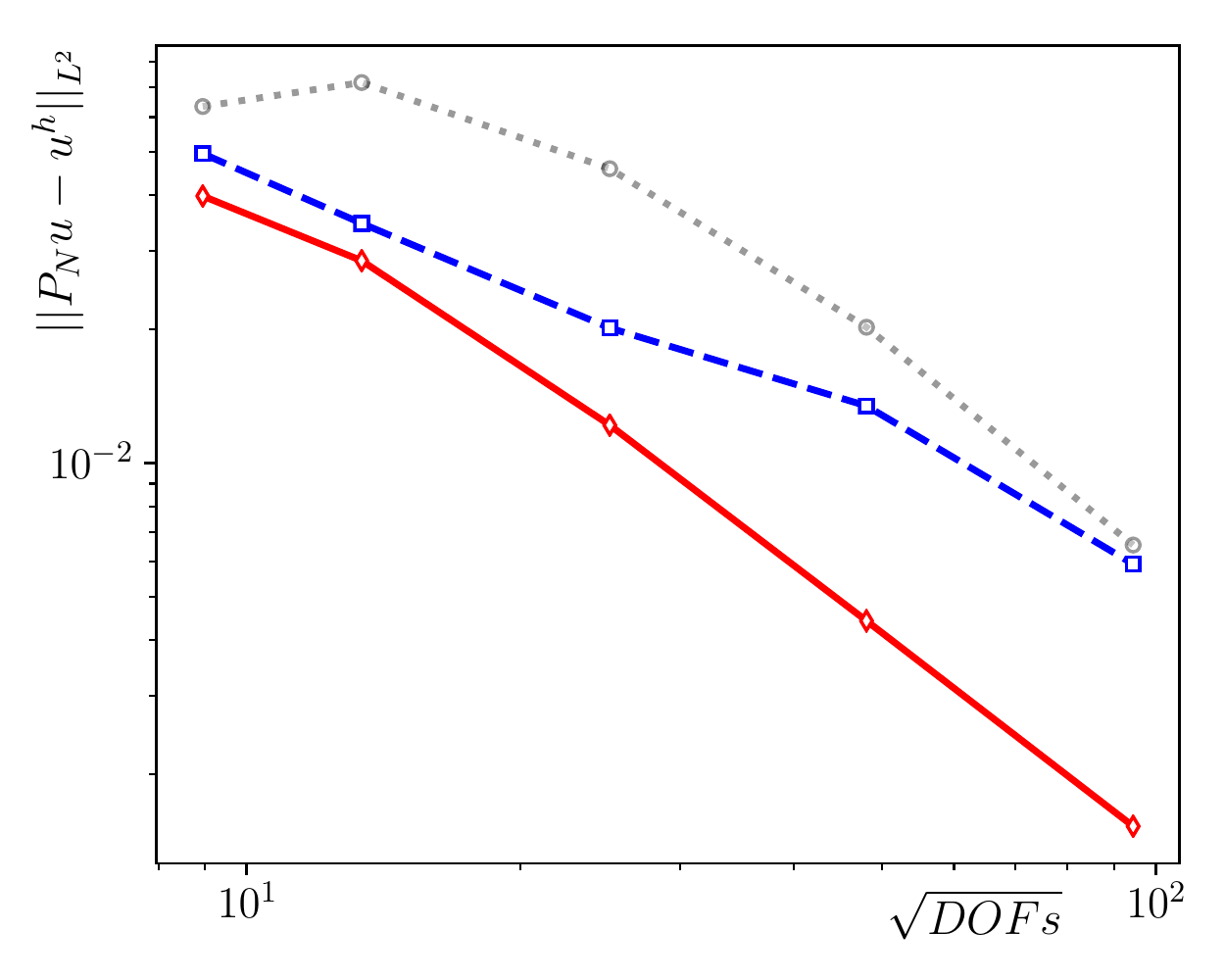} \label{fig:2DL2P1high}}
    \caption{Error with respect to the exact coarse-scale solution, using linear basis functions.}
    \label{fig:2DConvLin}
\end{figure}

\textbf{Remark 7:} Analogous to the one-dimensional case, the augmented VMS term decreases the diffusivity in the symmetry part of the Nitsche formulation. Different from the one-dimensional case, this becomes vector-valued and the formulation becomes a streamline directed modified diffusion on the boundary. One could interpret this as a boundary equivalent of the streamline diffusion that the classical VMS terms revert to for the same case.\\ 

\subsection{Higher-order basis functions}

Next, we change the geometry to a square with a polygonal exclusion, as depicted in \Cref{fig:2DproblemHO}. An exact geometry representation can be achieved, which, for these higher-order basis functions, is important for accurately computing the boundary integrals for the Nitsche projection $\mathscr{P}_N\phi$. We focus on the advection dominated case of $|\B{a}| =0.8$ and $\kappa=0.003$, and we use $\beta = 4P^2/h$.

\Cref{fig:2DErrorP2a,fig:2DErrorP2b,fig:2DErrorP2c} show the fine-scale solutions for quadratic basis functions obtained with the classical VMS model, the augmented VMS model and the Nitsche projector respectively. 
We observe that the classical VMS model with the parameter estimation developed in \cref{ssec:OpEst,ssec:TauEst} already performs remarkably well.
The boundary layers are almost exclusively contained in a single row of elements.

When we add the augmented term in the VMS model, the obtained error field qualitatively more closely resembles the true fine-scale solution shown in \Cref{fig:2DErrorP2c}. We do, however, also observe some small oscillations. This is consistent with the decreased diffusion interpretation proposed in \textit{Remarks }6 and 7. 
A more detailed analysis of the resulting error confirms that the solution obtained with the augmented VMS model more closely resembles the true coarse-scale solution defined by the Nitsche projector. This is shown in \Cref{fig:2DP2Err}, where the resulting fine-scale solution is plotted along a cut-plane, as well as in \Cref{fig:2DL2P2}, which plots the $L^2$-error with respect to the true coarse-scale solution for various mesh densities.\\ 


\begin{figure}[H]
    \centering
    \subfloat[Problem specification.]{\includegraphics[width=0.3\linewidth]{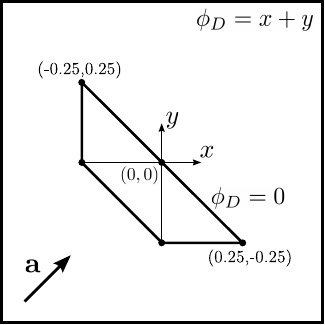}\label{fig:2DmodelHO}} \hspace{1.3cm}
    \subfloat[Solution for $|\B{a}| =0.8$ and $\kappa=0.003$.]{\includegraphics[width=0.47\linewidth]{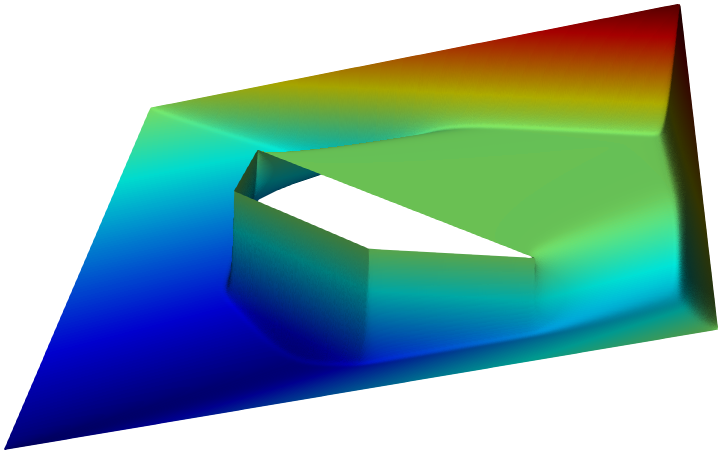} \label{fig:2DexactHO}}
    \caption{Two-dimensional model problem for higher-order basis functions.}
    \label{fig:2DproblemHO}
\end{figure}

\begin{figure}[H]
\vspace{-0.5cm}
    \centering
    \subfloat[Classical VMS model.] {\includegraphics[width=0.31\linewidth]{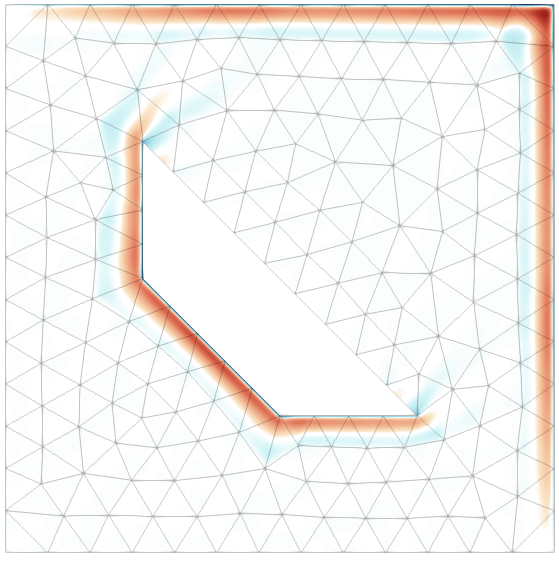}\label{fig:2DErrorP2a}} \hspace{0.2cm}
    \subfloat[Augmented VMS model.] {\includegraphics[width=0.31\linewidth]{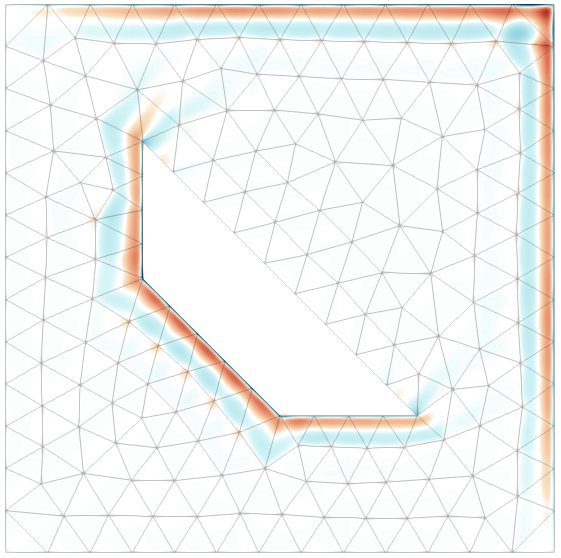} \label{fig:2DErrorP2b}}\hspace{0.2cm}
    \subfloat[Projected solution.] {\includegraphics[width=0.31\linewidth]{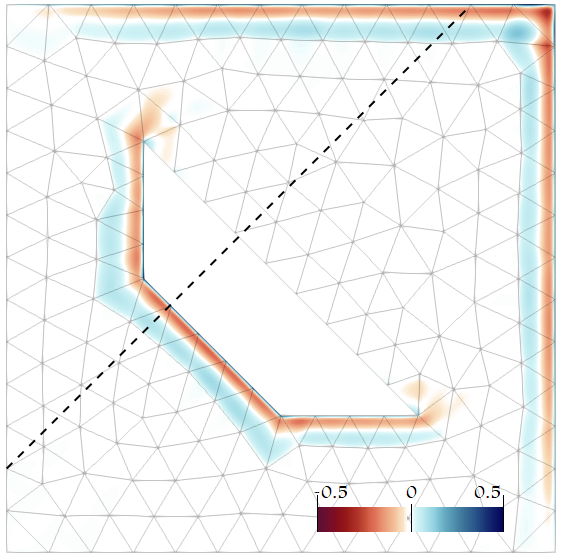} \label{fig:2DErrorP2c}}
    \caption{Fine-scale solutions $\phi-\phi^h$ (errors) for quadratic basis functions and $|\B{a}| =0.8$ and $\kappa=0.003$.}
    \label{fig:2DErrorP2}
\end{figure}

\begin{figure}[H]
    \centering
    \subfloat[Solutions on the cut-plane illustrated in \cref{fig:2DErrorP2c}.]{\includegraphics[width=0.57\linewidth,trim={0.3cm 0.35cm 0.3cm 0.3cm},clip]{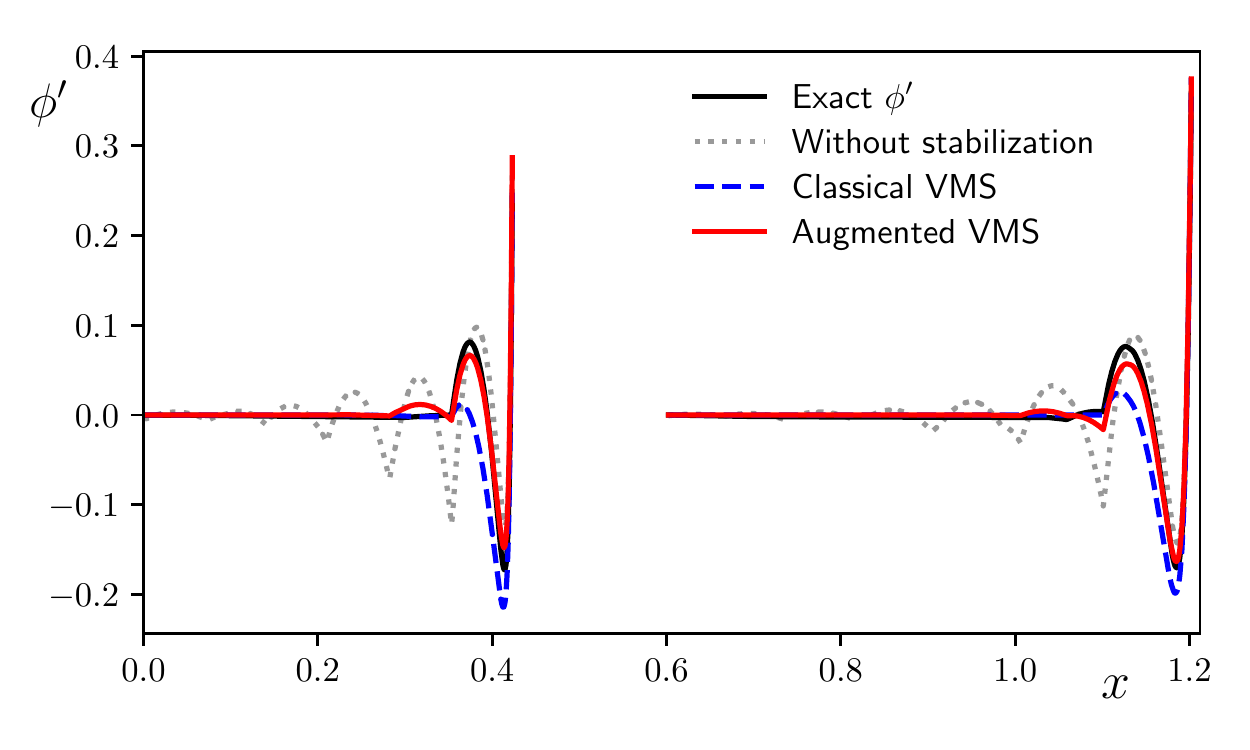} \label{fig:2DP2Err}}\hspace{0.2cm}
    \subfloat[Convergence with respect to exact coarse-scale solution.]{\includegraphics[width=0.39\linewidth,trim={0.4cm 0.0cm 0.4cm 0.4cm},clip]{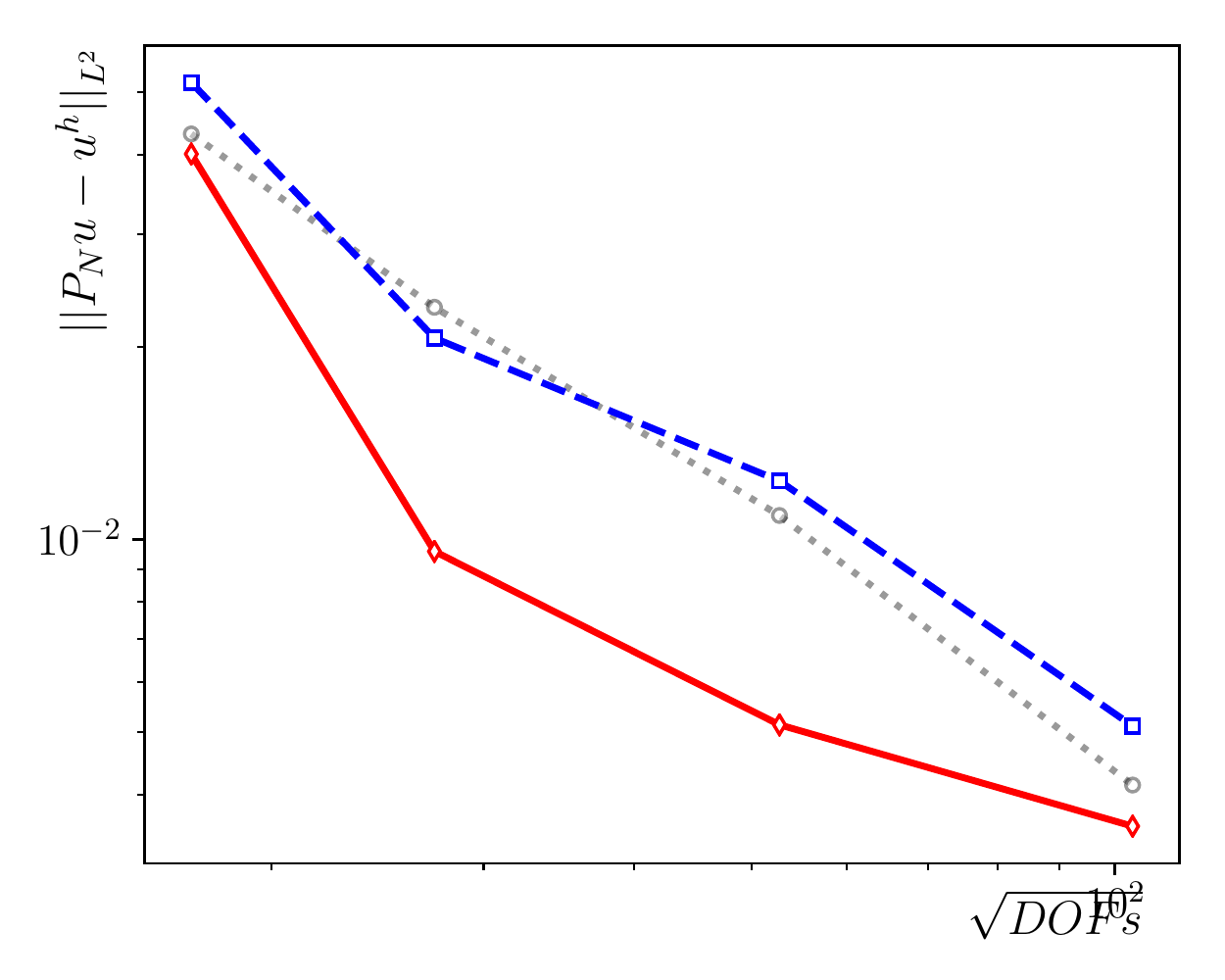}\label{fig:2DL2P2}} 
    \caption{Detailed error behavior for quadratic basis functions.}
    \label{fig:2DHOErr}
\end{figure}




Finally, if we use cubic basis functions, we obtain the results from \Cref{fig:2DErrorP3,fig:2DConvP}. Similar conclusions may be drawn as for the case of quadratic basis functions: the classical model with the parameters from \Cref{tab:cs} leads to a coarse-scale solution where the error is contained in the first row of elements. Adding the augmented term results in a solution that exhibits small oscillations, but nonetheless bears closer resemblance to the true coarse-scale solution, as measured qualitatively in \Cref{fig:2DP3Err} and quantitatively in \Cref{fig:2DL2P3}. 

It should also be noted that the Dirichlet boundary conditions are more closely satisfied with these cubic basis functions, as shown in \Cref{fig:2DP3Err}. This is, at least in part, due to the larger penalty parameter $\beta \propto P^2$. 
 The near-strong enforcement of the Dirichlet condition leaves a small fine-scale boundary value. The new term in the augmented model becomes almost inoperative, and the classical VMS model suffices. Indeed, the difference between the augmented and classical models is not as pronounced as it was in earlier simulations. 
These results convey that the augmented model provides fine-scale corrections in the pre-asymptotic regime, and vanishes (asymptotically) when such fine-scale corrections cease to be relevant.

\begin{figure}[H]
    \centering
    \subfloat[Classical VMS model.] {\includegraphics[width=0.31\linewidth]{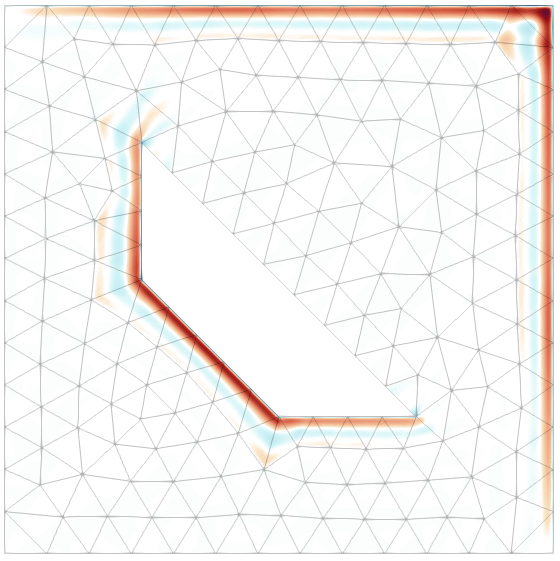}\label{fig:2DErrorP3a}} \hspace{0.2cm}
    \subfloat[Augmented VMS model.] {\includegraphics[width=0.31\linewidth]{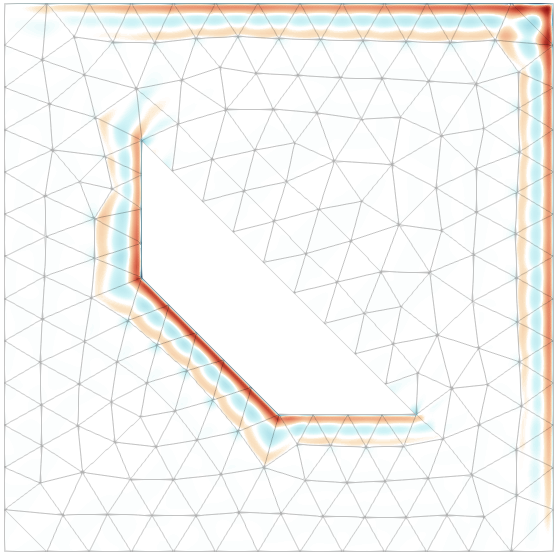} \label{fig:2DErrorP3b}}\hspace{0.2cm}
    \subfloat[Projected solution.] {\includegraphics[width=0.31\linewidth]{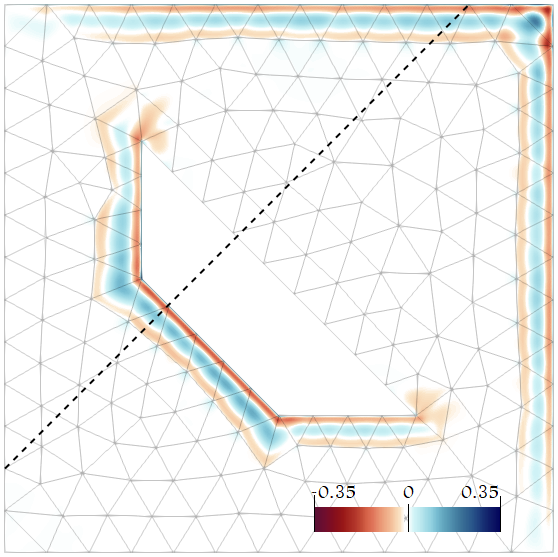} \label{fig:2DErrorP3c}}
    \caption{Fine-scale solutions $\phi-\phi^h$ (errors) for cubic basis functions and $|\B{a}| =0.8$ and $\kappa=0.003$.}
    \label{fig:2DErrorP3}
    \vspace{-0.5cm}
\end{figure}

\begin{figure}[H]
    \centering
    \subfloat[Solutions on the cut-plane illustrated in \cref{fig:2DErrorP3c}.]{\includegraphics[width=0.57\linewidth,trim={0.3cm 0.3cm 0.3cm 0.3cm},clip]{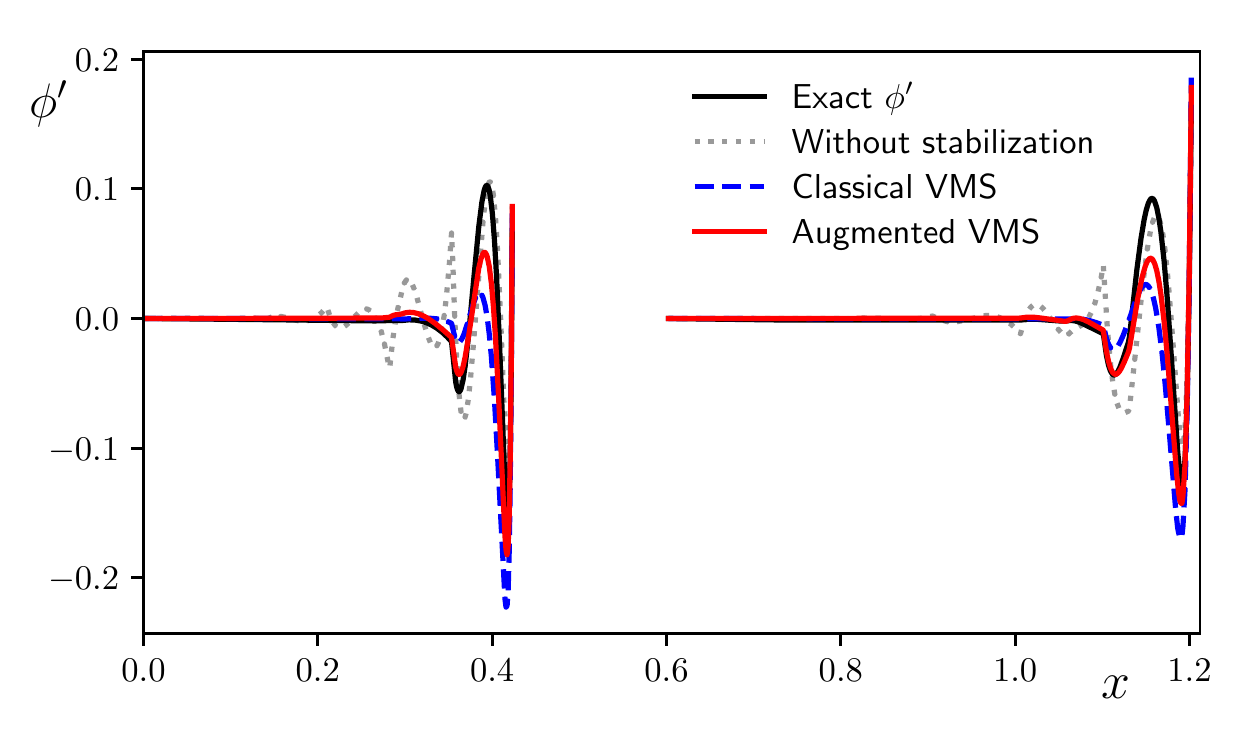} \label{fig:2DP3Err}}\hspace{0.2cm}
    \subfloat[Convergence with respect to exact coarse-scale solution.] {\includegraphics[width=0.39\linewidth,trim={0.4cm 0.0cm 0.4cm 0.4cm},clip]{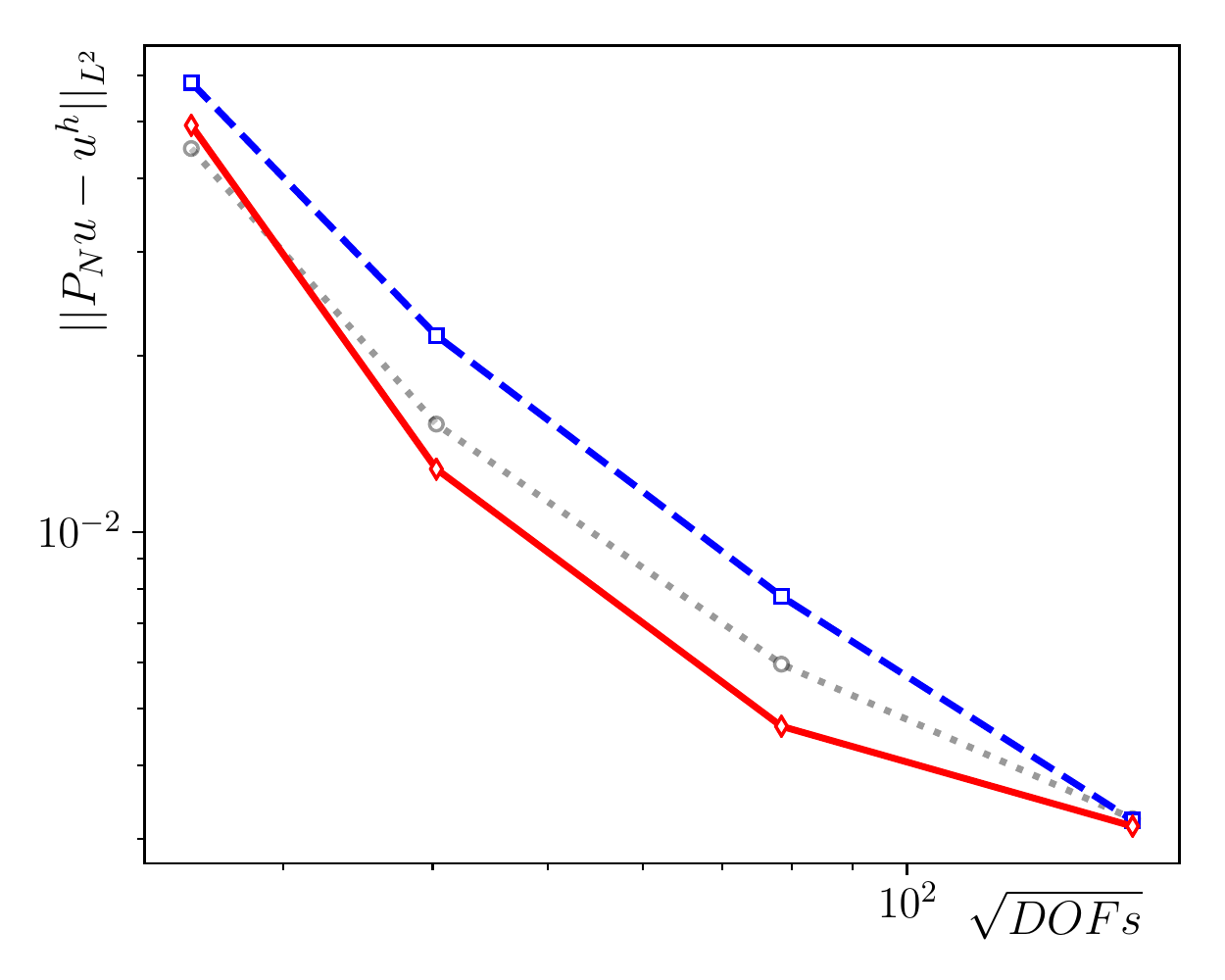} \label{fig:2DL2P3}}
    \caption{Detailed error behavior for cubic basis functions.}
    \label{fig:2DConvP}
\end{figure}



\section{Conclusion}
\label{sec:concl}

In this article, we unify the theories of variational multiscale analysis and weakly enforced boundary conditions into one consistent framework. Individually, these elemental numerical methods have shown great value in the context of fluid mechanics. With their merger, we are in a position to develop a fine-scale model that is appropriate for use in combination with Nitsche's method.

When the Dirichlet boundary conditions are enforced weakly, the standard $H^1_0$ projector is no longer applicable for the scale decomposition around which the variational multiscale method revolves. 
Instead, we propose a new projector, which we call the Nitsche projector. We show that adoption of this projector in the multiscale formulation naturally leads to Nitsche's formulation. That is, both the penalty term and the symmetry term in Nitsche's formulation automatically fall into place as part of the variational multiscale scale decomposition.

The model for the remaining fine-scale terms is based on the inversion of the fine-scale problem, which, in turn, is formally posed in the kernel space of the projector. We show that the functional constraints that define the $H^1_0$ projector are a subset of those corresponding to the Nitsche projector, such that we can largely base the inversion of the fine-scale problem on existing theory. An important difference in the context of weakly enforced boundary conditions is that the assumption of vanishing fine scales on element boundaries is no longer applicable for elements adjacent to the Dirichlet boundary. As a result, the fine-scale model that we obtain is the classical VMS model plus an additional boundary term. This `augmented' term takes into account the non-vanishing fine scales on the Dirichlet boundary. It may be interpreted as a consistent streamline diffusion in the symmetry term of Nitsche's formulation. 

Additionally, we develop approximations for the modeling terms $\tau$ and $\gamma$ based on the fine-scale Green's functions. These expressions and approximation strategies are also suitable for discretization with higher-order elements.

With this new model, and these new parameter definitions, we retrieve nodally exact solutions on one-dimensional meshes for all polynomial orders. This is an important property of the classical VMS model, which is lost when the boundary conditions are enforced weakly. On two-dimensional domains, we observe that the augmented model more closely resembles the actual coarse-scale solution defined by the Nitsche projector, as measured in an $L^2$ sense. This holds for all polynomial orders. For quadratic and cubic basis functions, the model without the augmented term already performs very well with the newly developed $\tau$ approximations. The error due to the boundary layer is contained in a single row of elements. 
For linear basis functions, however, the classical VMS model leads to a too thick boundary layer. This is almost completely mitigated when the augmented model is added to the formulation.\\

\noindent \textbf{Acknowledgments.} D. Schillinger gratefully acknowledges support from the National Science Foundation via the NSF CAREER Award No. 1651577 and from the German Research Foundation (Deutsche Forschungsgemeinschaft DFG) via the Emmy Noether Award \mbox{SCH 1249/2-1}. M.F.P. ten Eikelder and I. Akkerman are grateful for the support of Delft University of Technology.\\[0.36cm]

\appendix
\section{Fine-scale Green's functions and $\gamma$'s vanishing moments}\label{sec:fine-scale green}

In this appendix we draw conclusions on the $\gamma$ expression for the fine-scale Green's functions corresponding to $P=2$ and $P=3$ polynomial basis functions. Recall the definition of $\gamma$: 
\begin{align}
&\gamma = \frac{1}{|F|} \int\limits_{K}\!\int\limits_{F} \frac{x^{P-1}}{h^{P-1}}  \mathbb{H} g'^P_{H^1_0}(x,y)\d{y} \d{x} \,, \label{Agamdef} 
\end{align}
where we now add the superscript $P$ to the fine-scale Green's function to denotes the polynomial order of the coarse-scale basis functions on which $\mathscr{P}_{\!H^1_0}$ projects.

The fine-scale Green's function associated to the $H^1_0$-projector has been studied extensively in \cite{Hughes2007}. The authors prove the element local nature of $g'_{H^1_0}(x,y)$ in the one-dimensional case. They also show that in a single element, the fine-scale Green's function can then be obtained from the element local classical Green's function:
\begin{align}
\begin{split}
&   g'^P_{H^1_0}(x,y) = g(x,y)- \left[ \begin{matrix}  \int\limits_{0}^h g(x,y) \d y &\cdots& \int\limits_{0}^h y^{P-2}g(x,y) \d y \end{matrix}\right]\\ 
&   \hspace{0.5cm}\left[ \begin{matrix} 
    \int\limits_{0}^h\int\limits_{0}^h g(x,y)\d x\d y & \cdots &\int\limits_{0}^h\int\limits_{0}^h y^{P-2} g(x,y)\d x\d y \\
    \vdots & \ddots & \vdots \\
    \int\limits_{0}^h\int\limits_{0}^h x^{P-2} g(x,y)\d x\d y & \cdots &\int\limits_{0}^h\int\limits_{0}^h x^{P-2}y^{P-2} g(x,y)\d x\d y
    \end{matrix}\right]^{-1} \left[ \begin{matrix}  \int\limits_{0}^h g(x,y) \d x \\\vdots\\ \int\limits_{0}^h x^{P-2}g(x,y) \d x \end{matrix}\right] 
\end{split}. \label{gprime}
\end{align}
Refer to \cite{Stoter2017a} or \cite{Hughes2004b} for the expression for $g(x,y)$. The resulting functions for $P=1,2$ and $3$ are plotted in \Cref{fig:fsg}. For this particular case, the fine-scale Green's function for $P=1$ is exactly the element local classical Green's function $g(x,y)$.
\begin{figure}
    \centering
    \subfloat[$g'^1_{H^1_0}(x,y) = g(x,y)$.]{\includegraphics[width=0.32\linewidth]{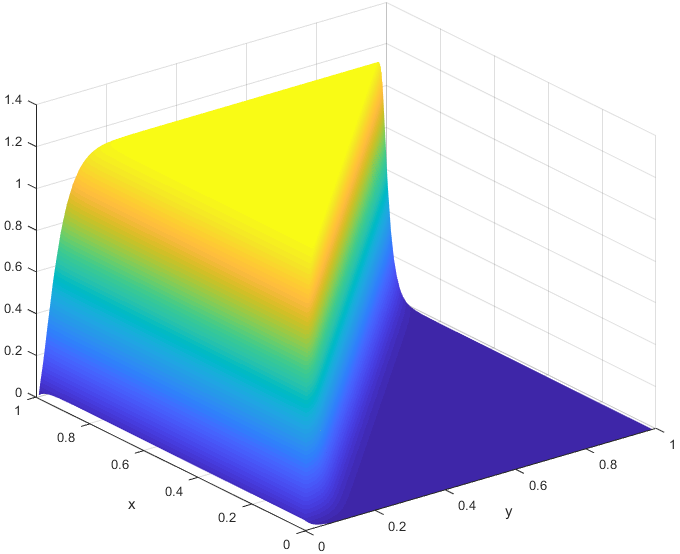}}\hspace{0.1cm}
    \subfloat[$g'^2_{H^1_0}(x,y)$.]{\includegraphics[width=0.32\linewidth]{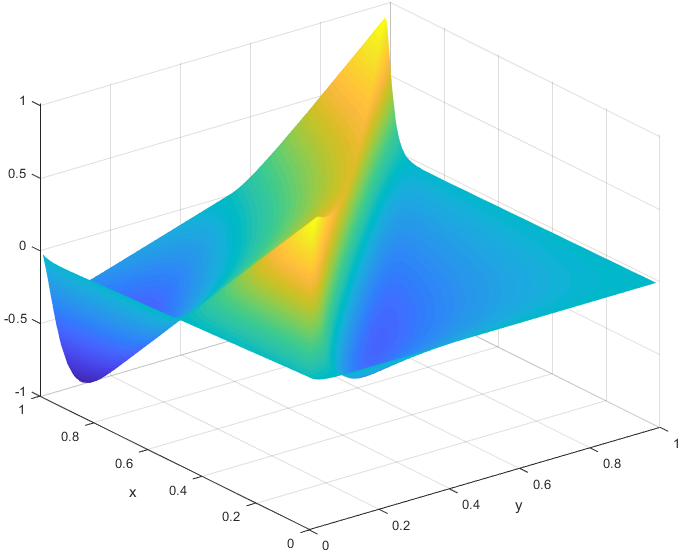}}\hspace{0.1cm}
    \subfloat[$g'^3_{H^1_0}(x,y)$.]{\includegraphics[width=0.32\linewidth]{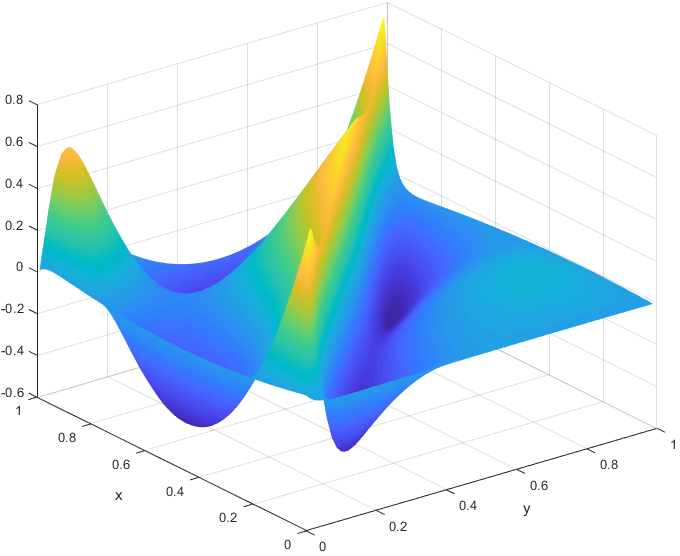}}
    \caption{Fine-scale Green's functions on one element for different polynomial coarse-scale basis functions. Using $\kappa = 0.02$ and $a=0.8$ on an element of size $h=1$.}
    \label{fig:fsg}
\end{figure}

The derivation in \cref{ssec:OneDModel} requires vanishing `moments' of the fine-scale Green's function, as stated in \cref{moment2}. We prove that this holds for $g'^P_{H^1_0}(x,y)$ in the following theorem.

\begin{thm}
\label{thm:0gamma} 
Define a $\gamma$-like parameter that depends on coarse-scale polynomial order $P$ and a $Q$th moment:
\begin{align}
\gamma^{Q,P } := \frac{1}{|F|} \int\limits_{0}^h\int\limits_{F} \frac{x^{Q-1}}{h^{Q-1}} \mathbb{H}  g'^P_{H^1_0}(x,y) \, \textup{d} y\, \textup{d} x \,,
\label{App1gammQP}
\end{align}
then from the definition of $ g'^P_{H^1_0}(x,y)$ in \cref{gprime} it follows that:
\begin{align}
\gamma^{Q,P } = \begin{cases}
0 \qquad & \text{if }Q<P \,,\\
\gamma\qquad & \text{if }Q=P \,.
\end{cases}
\end{align}
\end{thm}
\begin{proof}
The equality $\gamma^{Q,P } = \gamma$ for $Q=P$ follows directly from the definition of $\gamma$. For $Q<P$ we substitute the definition of the fine-scale Green's function from \cref{gprime}. After carrying out the integration in \cref{gprime}, the first vector becomes independent of the $y$-variable, and the last vector independent of the $x$-variable. The center matrix is filled with constants. This means that the differential operator and integration from \cref{App1gammQP} act on different vectors and they can thus be separated. After re-ordering of derivatives and integrals we obtain:
\begin{align}
\begin{split}
\gamma^{Q,P } = &\gamma^{Q,1} - \left[ \begin{matrix} \int\limits_{0}^h  \int\limits_{0}^h x^{Q-1} g(x,y) \d x \d y &\cdots&\int\limits_{0}^h  \int\limits_{0}^h  x^{Q-1} y^{P-2}g(x,y) \d x \d y \end{matrix}\right] \\ 
&   \hspace{0.5cm}\left[ \begin{matrix} 
    \int\limits_{0}^h\int\limits_{0}^h g(x,y)\d x\d y & \cdots &\int\limits_{0}^h\int\limits_{0}^h y^{P-2} g(x,y)\d x\d y \\
    \vdots & \ddots & \vdots \\
    \int\limits_{0}^h\int\limits_{0}^h x^{P-2} g(x,y)\d x\d y & \cdots &\int\limits_{0}^h\int\limits_{0}^h x^{P-2}y^{P-2} g(x,y)\d x\d y
    \end{matrix}\right]^{-1}
    \left[ \begin{matrix} \,\\[-0.2cm] \gamma^{1,1}  \\[0.3cm]\vdots\\[0.3cm]  \gamma^{P-2,1} \\[0.2cm] \end{matrix}\right].
\end{split}\label{AppTh2}
\end{align}
For ease of notation we denote the involved vectors and matrix $\boldsymbol{\zeta}^T$, $\textbf{C}^{-1}$ and $\boldsymbol{\xi}$. By recognizing that $\boldsymbol{\zeta}^T$ is the $Q$th row of $\textbf{C}$ we can write $\boldsymbol{\zeta}^T= \textbf{e}_Q^T \textbf{C}$, where $\textbf{e}_Q$ is a vector of zeros with a 1 at the $Q$th row. Substitution into the matrix-vector multiplication yields:
\begin{align}
    \gamma^{Q,P }_F = \gamma^{Q,1}_F -\boldsymbol{\zeta}^T\textbf{C}^{-1}\boldsymbol{\xi} = \gamma^{Q,1}_F - \textbf{e}_Q^T \textbf{C}\, \textbf{C}^{-1}\boldsymbol{\xi} = \gamma^{Q,1}_F - \textbf{e}_Q^T \boldsymbol{\xi} = \gamma^{Q,1}_F - \gamma^{Q,1}_F = 0 \,.
\end{align}
Note that this only holds for $Q<P$, since $\textbf{C}$ has $P-1$ rows.
\end{proof}

\bibliographystyle{unsrt}
\bibliography{MyBib}

\begin{thebibliography}{10}

\bibitem{Hughes1995}
T.J.R. {Hughes}.
\newblock {Multiscale phenomena: Green's functions, the Dirichlet--to--Neumann
  formulation, subgrid scale models, bubbles and the origins of stabilized
  methods}.
\newblock {\em Computer Methods in Applied Mechanics and Engineering},
  127(1-4):387--401, 1995.

\bibitem{Hughes1996}
T.J.R. Hughes and J.R. Stewart.
\newblock {A space--time formulation for multiscale phenomena}.
\newblock {\em Journal of Computational and Applied Mathematics},
  74(1-2):217--229, 1996.

\bibitem{Hughes1998}
T.J.R. Hughes, G.R. Feij\'oo, L.~{Mazzei}, and J.-B. Quincy.
\newblock {The variational multiscale method -- a paradigm for computational
  mechanics}.
\newblock {\em Computer Methods in Applied Mechanics and Engineering},
  166(1-2):3--24, 1998.

\bibitem{Franca2006}
L.P. Franca, G.~Hauke, and A.~Masud.
\newblock Revisiting stabilized finite element methods for the
  advective--diffusive equation.
\newblock {\em Computer Methods in Applied Mechanics and Engineering},
  195(13-16):1560--1572, 2006.

\bibitem{Coley2018}
C.~Coley and J.A. Evans.
\newblock Variational multiscale modeling with discontinuous subscales:
  analysis and application to scalar transport.
\newblock {\em Meccanica}, 53(6):1241--1269, 2018.

\bibitem{Bazilevs2007}
Y.~Bazilevs, V.M. Calo, J.A. Cottrell, T.J.R. Hughes, A.~Reali, and
  G.~Scovazzi.
\newblock {Variational multiscale residual-based turbulence modeling for large
  eddy simulation of incompressible flows}.
\newblock {\em Computer Methods in Applied Mechanics and Engineering},
  197(1-4):173--201, 2007.

\bibitem{Brezzi1997b}
F.~Brezzi, L.P. Franca, T.J.R. Hughes, and A.~Russo.
\newblock $ b=\int g $.
\newblock {\em Computer Methods in Applied Mechanics and Engineering},
  145(3-4):329--339, 1997.

\bibitem{Hughes2004b}
T.J.R. Hughes, G.~Scovazzi, and L.P. Franca.
\newblock Multiscale and stabilized methods.
\newblock In E.~Stein, R.~De~Borst, and T.J.R. Hughes, editors, {\em
  Encyclopedia of computational mechanics}, chapter~4. John Wiley \& Sons, Ltd,
  2004.

\bibitem{Brooks1982}
A.N. Brooks and T.J.R. Hughes.
\newblock {Streamline upwind/Petrov--Galerkin formulations for convection
  dominated flows with particular emphasis on the incompressible Navier--Stokes
  equations}.
\newblock {\em Computer Methods in Applied Mechanics and Engineering},
  32(1-3):199--259, 1982.

\bibitem{tenEikelder2018}
M.F.P. ten Eikelder and I.~Akkerman.
\newblock {Correct energy evolution of stabilized formulations: The relation
  between VMS, SUPG and GLS via dynamic orthogonal small-scales and
  isogeometric analysis. I: The convective--diffusive context}.
\newblock {\em Computer Methods in Applied Mechanics and Engineering},
  331:259--280, 2018.

\bibitem{HUGHES1989173}
T.J.R. Hughes, L.P. Franca, and G.M. Hulbert.
\newblock {A new finite element formulation for computational fluid dynamics:
  VIII. The Galerkin/least--squares method for advective--diffusive equations}.
\newblock {\em Computer Methods in Applied Mechanics and Engineering},
  73(2):173--189, 1989.

\bibitem{Tezduyar1991}
T.E. Tezduyar.
\newblock {Stabilized finite element formulations for incompressible flow
  computations}.
\newblock {\em Advances in Applied Mechanics}, 28(C):1--44, 1991.

\bibitem{Codina2007}
R.~Codina, J.~Principe, O.~Guasch, and S.~Badia.
\newblock {Time dependent subscales in the stabilized finite element
  approximation of incompressible flow problems}.
\newblock {\em Computer Methods in Applied Mechanics and Engineering},
  196(21):2413--2430, 2007.

\bibitem{Masud2006}
A.~Masud and R.A. Khurram.
\newblock {A multiscale finite element method for the incompressible
  Navier--Stokes equations}.
\newblock {\em Computer Methods in Applied Mechanics and Engineering},
  195(13-16):1750--1777, 2006.

\bibitem{Chang2012}
K.~Chang, T.J.R. Hughes, and V.M. Calo.
\newblock {Isogeometric variational multiscale large-eddy simulation of
  fully-developed turbulent flow over a wavy wall}.
\newblock {\em Computers \& Fluids}, 68:94--104, 2012.

\bibitem{Wang2010}
Z.~Wang and A.A. Oberai.
\newblock A mixed large eddy simulation model based on the residual-based
  variational multiscale formulation.
\newblock {\em Physics of Fluids}, 22(7):075107, 2010.

\bibitem{Gravemeier2011}
V.~Gravemeier and W.A. Wall.
\newblock Residual-based variational multiscale methods for laminar,
  transitional and turbulent variable-density flow at low mach number.
\newblock {\em International Journal for Numerical Methods in Fluids},
  65(10):1260--1278, 2011.

\bibitem{Takizawa2014}
K.~Takizawa, T.E. Tezduyar, S.~McIntyre, N.~Kostov, R.~Kolesar, and
  C.~Habluetzel.
\newblock {Space--time VMS computation of wind-turbine rotor and tower
  aerodynamics}.
\newblock {\em Computational Mechanics}, 53(1):1--15, 2014.

\bibitem{tenEikelder2018ii}
M.F.P. ten Eikelder and I.~Akkerman.
\newblock {Correct energy evolution of stabilized formulations: The relation
  between VMS, SUPG and GLS via dynamic orthogonal small-scales and
  isogeometric analysis. II: The incompressible Navier-Stokes equations}.
\newblock {\em Computer Methods in Applied Mechanics and Engineering},
  340:1135--1154, 2018.

\bibitem{Bazilevs2007weak_a}
Y.~Bazilevs and T.J.R. Hughes.
\newblock {Weak imposition of Dirichlet boundary conditions in fluid
  mechanics}.
\newblock {\em Computers \& Fluids}, 36(1):12--26, 2007.

\bibitem{Bazilevs2007weak_b}
Y.~Bazilevs, C.~Michler, V.M. Calo, and T.J.R. Hughes.
\newblock {Weak Dirichlet boundary conditions for wall-bounded turbulent
  flows}.
\newblock {\em Computer Methods in Applied Mechanics and Engineering},
  196(49-52):4853--4862, 2007.

\bibitem{bazilevs2010isogeometric}
Y.~Bazilevs, C.~Michler, V.M. Calo, and T.J.R. Hughes.
\newblock Isogeometric variational multiscale modeling of wall-bounded
  turbulent flows with weakly enforced boundary conditions on unstretched
  meshes.
\newblock {\em Computer Methods in Applied Mechanics and Engineering},
  199(13-16):780--790, 2010.

\bibitem{Nitsche1971}
J.~Nitsche.
\newblock {{\"U}ber ein Variationsprinzip zur L{\"o}sung von
  Dirichlet-Problemen bei Verwendung von Teilr{\"a}umen, die keinen
  Randbedingungen unterworfen sind}.
\newblock In {\em Abhandlungen aus dem mathematischen Seminar der
  Universit{\"a}t Hamburg}, volume~36, pages 9--15. Springer, 1971.

\bibitem{Burman2012}
E.~Burman.
\newblock {A penalty-free nonsymmetric Nitsche-type method for the weak
  imposition of boundary conditions}.
\newblock {\em SIAM Journal on Numerical Analysis}, 50(4):1959--1981, 2012.

\bibitem{bazilevs2012isogeometric}
Y.~Bazilevs, M.-C. Hsu, and M.A. Scott.
\newblock Isogeometric fluid--structure interaction analysis with emphasis on
  non-matching discretizations, and with application to wind turbines.
\newblock {\em Computer Methods in Applied Mechanics and Engineering},
  249:28--41, 2012.

\bibitem{xu2016tetrahedral}
F.~Xu, D.~Schillinger, D.~Kamensky, V.~Varduhn, C.~Wang, and M.-C. Hsu.
\newblock The tetrahedral finite cell method for fluids: Immersogeometric
  analysis of turbulent flow around complex geometries.
\newblock {\em Computers \& Fluids}, 141:135--154, 2016.

\bibitem{hsu2016direct}
M.-C. Hsu, C.~Wang, F.~Xu, A.J. Herrema, and A.~Krishnamurthy.
\newblock {Direct immersogeometric fluid flow analysis using B-rep CAD models}.
\newblock {\em Computer Aided Geometric Design}, 43:143--158, 2016.

\bibitem{kamensky2015immersogeometric}
D.~Kamensky, M.-C. Hsu, D.~Schillinger, J.A. Evans, A.~Aggarwal, Y.~Bazilevs,
  M.S. Sacks, and T.J.R. Hughes.
\newblock An immersogeometric variational framework for fluid--structure
  interaction: Application to bioprosthetic heart valves.
\newblock {\em Computer methods in applied mechanics and engineering},
  284:1005--1053, 2015.

\bibitem{wu2017optimizing}
M.C.H. Wu, D.~Kamensky, C.~Wang, A.J. Herrema, F.~Xu, M.S. Pigazzini, A.~Verma,
  A.L. Marsden, Y.~Bazilevs, and M.-C. Hsu.
\newblock Optimizing fluid--structure interaction systems with immersogeometric
  analysis and surrogate modeling: Application to a hydraulic arresting gear.
\newblock {\em Computer Methods in Applied Mechanics and Engineering},
  316:668--693, 2017.

\bibitem{hsu2014fluid}
M.-C. Hsu, D.~Kamensky, Y.~Bazilevs, M.S. Sacks, and T.J.R. Hughes.
\newblock Fluid--structure interaction analysis of bioprosthetic heart valves:
  significance of arterial wall deformation.
\newblock {\em Computational mechanics}, 54(4):1055--1071, 2014.

\bibitem{hoang2019}
T.~Hoang, C.V. Verhoosel, C.-Z. Qin, F.~Auricchio, A.~Reali, and E.H. van
  Brummelen.
\newblock Skeleton-stabilized immersogeometric analysis for incompressible
  viscous flow problems.
\newblock {\em Computer Methods in Applied Mechanics and Engineering},
  344:421--450, 2019.

\bibitem{Stoter2017a}
S.K.F. {Stoter}, S.R. {Turteltaub}, S.J. {Hulshoff}, and D.~{Schillinger}.
\newblock {Residual-based variational multiscale modeling in a discontinuous
  Galerkin framework}.
\newblock {\em Multiscale Modeling and Simulation: A SIAM Interdisciplinary
  Journal}, 16(3):1333--1364, 2018.

\bibitem{Stoter2017b}
S.K.F. {Stoter}, S.R. {Turteltaub}, S.J. {Hulshoff}, and D.~{Schillinger}.
\newblock {A discontinuous Galerkin residual-based variational multiscale
  method for modeling subgrid-scale behavior of the viscous Burgers equation}.
\newblock {\em International Journal for Numerical Methods in Fluids},
  88(5):217--238, 2018.

\bibitem{Stoter2019a}
S.K.F. {Stoter}, B.~{Cockburn}, and D.~{Schillinger}.
\newblock {Mixed and discontinuous Galerkin methods through the lens of
  variational multiscale analysis}.
\newblock {\em \text{Submitted to:} Computer Methods in Applied Mechanics and
  Engineering}, 2018.

\bibitem{Hughes2007}
T.J.R. Hughes and G.~Sangalli.
\newblock {Variational multiscale analysis: The fine-scale Green's function,
  projection, optimization, localization, and stabilized methods}.
\newblock {\em Society for Industrial and Applied Mathematics}, 45(2):539--557,
  2007.

\bibitem{Harari2018}
I.~Harari and U.~Albocher.
\newblock {Spectral investigations of Nitsche's method}.
\newblock {\em Finite Elements in Analysis and Design}, 145:20--31, 2018.

\bibitem{harari1992c}
I.~Harari and T.J.R. Hughes.
\newblock {What are C and h?: Inequalities for the analysis and design of
  finite element methods}.
\newblock {\em Computer Methods in Applied Mechanics and Engineering},
  97(2):157--192, 1992.

\bibitem{Hughes1986}
T.J.R. Hughes and M.~Mallet.
\newblock {A new finite element formulation for computational fluid dynamics:
  III. The generalized streamline operator for multidimensional
  advective-diffusive systems}.
\newblock {\em Computer Methods in Applied Mechanics and Engineering},
  58(3):305--328, 1986.

\bibitem{Shakib1991}
F.~Shakib, T.J.R. Hughes, and Z.~Johan.
\newblock {A new finite element formulation for computational fluid dynamics:
  X. The compressible Euler and Navier-Stokes equations}.
\newblock {\em Computer Methods in Applied Mechanics and Engineering},
  89(1-3):141--219, 1991.

\bibitem{Tezduyar2000}
T.E. Tezduyar and Y.~Osawa.
\newblock {Finite element stabilization parameters computed from element
  matrices and vectors}.
\newblock {\em Computer Methods in Applied Mechanics and Engineering},
  190(3-4):411--430, 1999.

\bibitem{Embar2010}
A.~Embar, J.~Dolbow, and I.~Harari.
\newblock Imposing {D}irichlet boundary conditions with {N}itsche's method and
  spline based finite elements.
\newblock {\em International Journal for Numerical Methods in Engineering},
  83(7):877--898, 2010.

\end{thebibliography}

\end{document}